\documentclass{tran-l}
\usepackage{amssymb}
\usepackage[T1]{fontenc}
\usepackage{booktabs}
\usepackage{algorithm}
\usepackage{algpseudocode}
\usepackage{threeparttable}
\usepackage{tikz}
\usepackage{xcolor}
\usepackage[hidelinks]{hyperref}
\usepackage[capitalise,nameinlink]{cleveref}

\newtheorem{theorem}{Theorem}[section]
\newtheorem{lemma}[theorem]{Lemma}
\newtheorem{corollary}[theorem]{Corollary}
\newtheorem{proposition}[theorem]{Proposition}

\theoremstyle{definition}
\newtheorem{definition}[theorem]{Definition}
\newtheorem{example}[theorem]{Example}

\theoremstyle{remark}
\newtheorem{remark}[theorem]{Remark}

\numberwithin{equation}{section}
\numberwithin{algorithm}{section}

\crefname{theorem}{Theorem}{Theorems}
\crefname{lemma}{Lemma}{Lemmas}
\crefname{corollary}{Corollary}{Corollaries}
\crefname{proposition}{Proposition}{Propositions}
\crefname{definition}{Definition}{Definitions}
\crefname{example}{Example}{Examples}
\crefname{remark}{Remark}{Remarks}
\crefname{section}{Section}{Sections}
\crefname{subsection}{Subsection}{Subsections}
\crefname{algorithm}{Algorithm}{Algorithms}
\crefname{table}{Table}{Tables}
\Crefname{subsection}{Subsection}{Subsections}

\hypersetup{
  pdftitle={Unitary-Invariant Decomposition of Reducible Total Least Squares Core Problems},
  pdfauthor={S. Yu, B. Carpentieri, and Y.-F. Jing}
}

\begin{document}

\title[Decomposition of Total Least Squares Core Problems]{Unitary-Invariant Decomposition of Reducible Total Least Squares Core Problems}

\author{Sijia Yu}
\address{Faculty of Engineering, Free University of Bozen--Bolzano, Italy}
\email{Sijia.Yu@student.unibz.it}
\thanks{Corresponding author: Sijia Yu.}

\author{Bruno Carpentieri}
\address{Faculty of Engineering, Free University of Bozen--Bolzano, Italy}
\email{Bruno.Carpentieri@unibz.it}

\author{Yan-Fei Jing}
\address{School of Mathematical Sciences / Institute of Computational Science,
University of Electronic Science and Technology of China, China}
\email{yanfeijing@uestc.edu.cn}



\begin{abstract}
The analysis of a total least squares problem can be reduced to that of an associated core problem, which typically has lower dimension and improved solubility properties. Nevertheless, even a core problem may remain reducible, admitting further decomposition into irreducible component subproblems with simpler structure and better analytical properties. We seek the finest decomposition under unitary transformations over
$\mathbb{C}$; for real data this decomposition can be strictly finer than the finest one under real orthogonal transformations. 
We split the data matrix into row blocks associated with the singular subspaces of the system matrix and consider the graph whose vertices are the blocks and whose edges correspond to nonzero cross-products. After preliminary splittings that make every nonzero cross-product a scalar multiple of a unitary matrix, we show that the components can be obtained directly from these blocks if and only if the covariance matrices along all cycles of the graph are scalar matrices. Otherwise, there is a chordless cycle whose covariance matrix is not a scalar matrix, and splitting the blocks on this cycle along the eigenspaces of the covariance matrix removes the obstruction without changing the decomposition. Based on these results, we propose a recursive strategy that finds all irreducible components over the complex field. The strategy is exact and theoretical rather than an efficient numerical algorithm; a small floating-point computation illustrates it. The results depend only on the core problem and therefore apply to general linear approximation problems
$AX\approx B$.
\end{abstract}

\maketitle

\section{Introduction}

In this paper, we consider an inconsistent linear problem of the form
\begin{equation}
AX \approx B,
\label{eq:incompatible problem}
\end{equation}
where $A \in \mathbb{F}^{m \times n}$ is a system (data) matrix, $B \in \mathbb{F}^{m \times d}$ is a matrix of observations, and $X \in \mathbb{F}^{n \times d}$ is the unknown solution matrix, with $\mathbb{F} \in \{\mathbb{R}, \mathbb{C}\}$. We assume that $A^*B \neq 0$ and that $\mathcal{R}(B) \not\subset \mathcal{R}(A)$, so that \eqref{eq:incompatible problem} admits no exact solution.

Such inconsistency is not exceptional but naturally arises in a wide range of applications in which both the system model and the observed data are affected by noise, modeling uncertainty, or measurement errors. Typical examples include system identification \cite{someidentificationmethods,systempulsetransfer}, signal processing \cite{Exactmaximumlikelihood}, and linear regression with errors in variables \cite{Estimation}, as commonly encountered in statistical learning and data-driven modeling. In these settings, the construction of reliable approximate solutions and the analysis of error propagation in the presence of incompatible data are of central importance, and have therefore attracted considerable attention in the literature \cite{Ananalysisof,ref4}.

Classical approaches to inconsistent linear problems include the least squares (LS) formulation \cite{ref4, EDS.Total, Algebraic}, the data least squares (DLS) problem \cite{Thedata, EDS.Total}, and the total least squares (TLS) problem \cite{Somemodified, Algebraic, Theanalysis}. While LS accounts for perturbations only in the observation matrix and DLS assumes errors solely in the data matrix, the TLS framework allows for simultaneous perturbations in both $A$ and $B$. For this reason, TLS provides a more realistic modeling paradigm in applications where uncertainty affects both the system and the measurements.

A substantial body of work has established fundamental results on the existence of TLS solutions and has led to the development of effective algorithms for computing exact or approximate, including truncated, TLS solutions \cite{Ananalysisof,Classicalworks,EDS.Total,OptimalBackward}. 
Most of these contributions, however, are primarily concerned with the computation of TLS solutions and place less emphasis on the internal structure of large-scale TLS problems.

A central idea in the classical analysis of TLS problems, and the main tool for dimension reduction, is the use of orthogonal transformations to isolate the essential part of the data \cite{Ananalysisof}. 
In this spirit, Paige and Strakoš introduced the notion of the TLS core problem, which yields a reduced subproblem retaining all the information required to solve the original TLS problem with a single right-hand side \cite{Scaledtotal}.
This concept was subsequently extended to multiple right-hand sides \cite{Thecore}, to complex-valued TLS problems \cite{TLSformulation}, and to bilinear and tensor formulations \cite{TLSformulation,OnTLS,KrylovSubspace}. 
The TLS core problem is unique up to orthogonal equivalence and can be computed efficiently using Krylov-based techniques \cite{BGGK}.

However, while the TLS core problem isolates the essential part of the data responsible for incompatibility, it does not by itself explain how this core may further decompose internally, nor how unsolvability is distributed across its substructures. These observations motivate the development of a finer structural theory for TLS core problems. 

As shown in \cite{TLSsense,minimization,AnalysisPropertiesandBehaviour}, certain classes of TLS core problems, referred to as \emph{reducible} core problems, admit further decomposition into independent component subproblems through suitable orthogonal transformations. 
In particular, component subproblems are typically lower-dimensional and structurally simpler, which makes them natural candidates for future solvability and numerical analysis. 
Even when a core problem has no TLS solution, some of its
components may be TLS solvable~\cite{TLSsense,minimization,AnalysisPropertiesandBehaviour}.

More recently, the notion of the \emph{row separation problem} was introduced as a key tool for identifying admissible decompositions of TLS core problems \cite{Possible}. 
Within this framework, component subproblems are associated with row separation subspaces, which can in turn be generated from indivisible subspaces. The strategy of~\cite{Possible} identifies component subproblems in the case of simple singular values, and the uniqueness of the resulting irreducible reduced structures was left open there. 
The case of multiple singular values is addressed here for $\mathbb{F}=\mathbb{C}$ by the strategy of \Cref{sec:division problem construction in complex number}. For the uniqueness question, we observe that the finest reduced structure corresponds to a finest simultaneous singular value decomposition of a pair of matrices formed from the core problem; its uniqueness, up to unitary equivalence and permutation of the components, therefore follows from the theory of Maehara and Murota~\cite{MaeharaMurota2011LAA} (\Cref{sec:ssvd}).

The aim of this paper is to find all components of a core problem under unitary transformations, that is, its finest unitary decomposition. For real data, this decomposition is at least as fine as the finest decomposition under real orthogonal transformations and may be strictly finer.
The strategy is formulated in exact arithmetic: each of its decisions is an
exact test for equality or for vanishing. It should therefore be regarded as a theoretical procedure, and not as an efficient numerical algorithm. In this sense, the present theory provides an exact structural reference against which future, more efficient algorithms can be verified.

The contributions of this paper are fivefold: (i) the extension of the row separation framework of~\cite{Possible} to complex problems; (ii) a characterization of row separation subspaces through the row-space Gram operators of the blocks, together with an equivalence between problems and a construction of equivalent problems by exchanging subproblems; (iii) general division problems over $\mathbb{R}$ and $\mathbb{C}$, with a characterization of complete division by cycle products; (iv) a strategy that finds all indivisible subspaces over $\mathbb{C}$, without restrictions on the multiplicities of the singular values and without prior knowledge of the sizes of the components; and (v) the correspondence between the finest reduced structure and a finest simultaneous singular value decomposition~\cite{MaeharaMurota2011LAA}, together with a description of how its multiplicity structure manifests itself in the row separation framework as complete division problems.

Although our motivation comes from TLS, reducibility is a property of the core problem itself and does not involve the TLS objective. The results of this paper therefore apply to the core problem of any linear approximation problem $AX\approx B$, whether it is posed in the TLS, LS or DLS sense.

At the matrix level, if $A_{11}X_{11}\approx B_1$ denotes the core problem, the structural question is to determine the finest unitary transformations $P$, $R$, $Q$ for which $P^*[B_1R\mid A_{11}Q]$ becomes block diagonal.
The target decomposition is closely related to simultaneous block
diagonalization: as shown in \Cref{sec:ssvd}, it corresponds to the
simultaneous block diagonalization of $B_1B_1^*$ and $A_{11}A_{11}^*$ and to a finest simultaneous singular value decomposition in the sense of Maehara and Murota~\cite{MaeharaMurota2011LAA}. General algorithms for simultaneous block diagonalization and for simultaneous singular value decompositions~\cite{MurotaKannoKojimaKojima2010,MaeharaMurota2010,MaeharaMurota2011LAA} can compute such component subproblems (see \Cref{sec:ssvd}), but they do not exploit the structure of core problems.
This paper provides 
a route to this decomposition, and a characterization of it, that are specific to core problems, through the graph structure induced by the orthogonality relations between row blocks and the spectral characterization
by covariance matrices along its cycles.

\subsection{Main ideas and roadmap}\label{sec:roadmap}
We summarize the main ideas informally, before the terminology of \Cref{sec:pre,sec:supplement,sec:division problem theory} is introduced.

\emph{Invariant subspaces.} Splitting the data matrix $B_1$ of a core problem according to the left singular subspaces of $A_{11}$ produces row blocks $D_1,\ldots,D_{k+1}$. Admissible decompositions of the core problem are encoded by row separation subspaces of these blocks, while the minimal ones, the indivisible subspaces, determine the irreducible components~\cite{Possible}. By \Cref{le:D_j and S cap}, the row separation subspaces are exactly the nonzero subspaces of the total row space $\mathcal{V}=\sum_j\mathcal{R}(D_j^T)$ that are invariant under all Hermitian row-space Gram operators $D_j^T\bar D_j$. Decomposing a core problem therefore means finding the minimal common invariant subspaces of a family of Hermitian matrices that has a particular structure, and the paper exploits this structure.

\emph{A graph and its cycles.} After the preliminary one-block and pairwise splittings have produced a $2$-division problem, each block has orthogonal rows of equal norm and every nonzero cross-product $\bar D_iD_j^T$ is a scalar multiple of a unitary matrix. The blocks are the vertices of a graph whose edges are the nonzero cross-products. Unitary changes of basis within the blocks reduce the cross-products along a spanning tree to scalar matrices, so the remaining information is carried by the products of cross-products around cycles, which, following~\cite{Possible}, we call covariance matrices. \Cref{thm:relation between eigen/singular} shows that the problem is a complete division problem exactly when all these cycle products are scalar matrices, while \Cref{thm:single cycle} shows that the first non-scalar obstruction can be chosen on a chordless cycle.

\emph{The strategy.} 
\Cref{alg:strategy} splits the blocks of such a cycle along the eigenspaces of its covariance matrix, without changing the row separation subspaces.
It terminates with all indivisible subspaces and hence with the finest unitary decomposition of the core problem, without prior knowledge of the number or sizes of the components. It is formulated in exact arithmetic; \Cref{sec:example} illustrates it by a floating-point computation with an explicit tolerance.

The rest of the paper is organized as follows.
\Cref{sec:pre} recalls core problems and row separation
problems. \Cref{sec:supplement} gives properties of row separation subspaces and the relations between row separation problems. \Cref{sec:division problem theory}
develops the theory of $c$-division problems.
\Cref{sec:division problem construction in complex number} presents the strategy, its relation to the simultaneous singular value decomposition and the structure of its output, and a floating-point illustration.
\Cref{sec:con} concludes the paper. 
\section{Preliminaries and Row Separation Framework}\label{sec:pre}

We begin by recalling the notion of a TLS core problem and summarizing its essential properties \cite{TLSformulation}, which provide the natural starting point for any intrinsic analysis of unsolvable TLS instances.
We then introduce the row separation framework and the associated subspaces \cite{Possible}, which supply the abstract language and mechanisms needed to identify admissible decompositions and irreducible components.
Together, these concepts form the theoretical foundation for the decomposition strategies to be developed in Sections~3–5.

\subsection{Notation and Conventions}\label{sec:notations}
Unless otherwise stated, all spaces and linear operations are considered over the field $\mathbb{F}\in\{\mathbb{R},\mathbb{C}\}$. Over $\mathbb{F}^{n}$ we use
the inner product $\langle x,y\rangle=y^{*}x$ and the norm
$\|x\|_2=\langle x,x\rangle^{1/2}$, where $x,y\in\mathbb{F}^{n\times 1}$; for $\mathbb{F}=\mathbb{R}$ this is the standard Euclidean inner product.

We write $A^{T}$ for the transpose of $A$, $\bar{A}$ for its entry-wise complex conjugate, and $A^{*}$ for its conjugate transpose. For a matrix $A$, $\|A\|_F$ and $\|A\|_2$ denote its Frobenius norm and its spectral norm, respectively. For a matrix $A$, $\mathcal{R}(A)$, $\mathcal{R}(A^{T})$ and $\mathcal{N}(A)$ respectively denote its column space, row space and null space over $\mathbb{F}$. 
The symbols $I$ and $0$ denote the identity matrix and the zero matrix, respectively. A square matrix of the form $\alpha I$ with $\alpha\in\mathbb{F}$ is called a scalar matrix, where $\alpha$ could be zero.

For a subspace $\mathcal{A}$, $\dim(\mathcal{A})$ denotes its dimension.
For subspaces $\mathcal{A}_1$ and $\mathcal{A}_2$, $+$ denotes their sum and $\oplus$ their direct sum; orthogonality of the summands is stated explicitly when required. If $\mathcal{A}_2\subseteq\mathcal{A}_1$, then
$\mathcal{A}_1-\mathcal{A}_2$ denotes the subspace spanned by the vectors in $\mathcal{A}_1$ that are orthogonal to every vector in $\mathcal{A}_2$, that is, the orthogonal complement of $\mathcal{A}_2$ in $\mathcal{A}_1$.

To distinguish between ordered and unordered collections of indices, we adopt the following notation. The symbol $(a_1,\ldots,a_k)$ denotes an ordered collection of $k$ elements. The symbol $\{C\}=\{a_1,\ldots,a_k\}$ denotes an unordered collection of $k$ distinct elements, and $c=|\{C\}|$ denotes its cardinality. An ordering of $\{C\}$ is an ordered collection $(j_1,\ldots,j_c)$ in which each element of $\{C\}$ appears exactly once.

\subsection{TLS Problem and Core Problem in $\mathbb{F}$}\label{sec:coreproblem}

We briefly recall the TLS formulation and introduce the notion of the TLS core problem, which serves as the starting point for the structural analysis developed in this paper.

Given a system matrix $A \in \mathbb{F}^{m \times n}$ and a right-hand side matrix $B \in \mathbb{F}^{m \times d}$, the TLS framework accounts for perturbations in both the data matrix and the observations. 
Specifically, it introduces perturbation matrices $E \in \mathbb{F}^{m \times n}$ and $G \in \mathbb{F}^{m \times d}$ and seeks to solve
\begin{equation}
\min_{E,G}\ \|[E \mid G]\|_F
\quad \text{subject to} \quad (A+E)X_{\mathrm{TLS}} = B+G,
\label{eq:TLS}
\end{equation}
where $X_{\mathrm{TLS}} \in \mathbb{F}^{n \times d}$ denotes a TLS solution.

The existence, uniqueness, and structural properties of TLS solutions have been extensively studied in the literature; see, for example, \cite{Classicalworks,TLSformulation} for classical solvability conditions and characterizations. 

Next we recall the notion of the TLS core problem, which provides a dimension-reduced representation that retains all structurally relevant information for the analysis of TLS problems.

\begin{definition}[Core Problem in $\mathbb{F}$ {\cite{BGGK,TLSformulation}}]\label{def:corepro}
To isolate the essential part of the approximation problem $AX \approx B$, consider unitary matrices
$P$, $Q$, and $R$, such that
\begin{equation}
\left[\widehat{B}\ \vert\ \widehat{A}\right]
= P^*\left[B\ \vert\ A\right]
\begin{bmatrix} R & 0 \\ 0 & Q \end{bmatrix}
= P^*\left[BR\ \vert\ AQ\right]
\equiv
\left[\begin{array}{c|c||c|c}
B_1 & 0 & A_{11} & 0\\
\hline
0 & 0 & 0 & A_{22}
\end{array}\right],
\label{eq2.5}
\end{equation}
where the block $[B_1 \mid A_{11}]$ has minimal dimension and the block $A_{22}$ has maximal dimension among all such unitary transformations.
The subproblem $A_{11}X_{11} \approx B_1$ is called the \emph{core problem}.
\end{definition}

\begin{remark} 
The detailed relationship between the TLS problem and its core problem, including how solutions of the original problem can be recovered from solutions of the core problem, is described in \cite{BGGK,TLSformulation,ReductionofData}.
\end{remark}

The core problem provides a natural starting point for the subsequent structural analysis. 
In particular, both the data matrix $A_{11}$ and the observation matrix $B_1$ are of full column rank; see \cite{TLSformulation} for details.

\begin{definition}[Reducible (Irreducible) Core Problem over $\mathbb{F}$, \cite{TLSsense} for $\mathbb{F}=\mathbb{R}$]\label{def:red}
A core problem $\left[B_1 \ \vert \ A_{11}\right]$ is called \emph{reducible} (or \emph{composed}) over $\mathbb{R}$ (resp.\ $\mathbb{C}$) if, after orthogonal (resp.\ unitary) transformations $P$, $Q$, $R$, it admits the block diagonal structure, with $g\geq 2$,
\begin{equation}
P^*\left[B_1 \ \vert\ A_{11}\right]
\begin{bmatrix} R & 0 \\ 0 & Q \end{bmatrix}
=
\left[\begin{array}{ccc|ccc}
B_1^{(1)} &&& A_{11}^{(1)} \\
& \ddots &&& \ddots \\
&& B_1^{(g)} &&& A_{11}^{(g)}
\end{array}\right],
\label{eq2.6}
\end{equation}
possibly including the case in which some block $A_{11}^{(t)}$ has no columns, for $t=1,\ldots,g$. A core problem that is not reducible over $\mathbb{F}$ is called irreducible over $\mathbb{F}$.

The individual subproblems $\left[B_1^{(t)} \ \vert\ A_{11}^{(t)}\right]$, $t=1,\ldots,g$, are called the \emph{components} of the core problem.  
If $\mathcal{R}\!\left(B_1^{(t)}\right) \subseteq \mathcal{R}\!\left(A_{11}^{(t)}\right)$, the component is called \emph{compatible}; otherwise, it is called \emph{incompatible}.  
If $A_{11}^{(t)}$ has no columns, the component $\left[B_1^{(t)}\right]$ is called \emph{degenerate}.
\end{definition}

\begin{remark}\label{rem:field}
Whether a core problem is reducible depends on the field $\mathbb{F}$. A real core problem that is irreducible over $\mathbb{R}$ may be reducible
over $\mathbb{C}$; see \Cref{def:dense problem} and Example~\ref{eg:C+R}.

\end{remark}

\begin{remark}
A block diagonal form \eqref{eq2.6} is called a
\emph{reduced structure} over $\mathbb{F}$ of the core problem. A reduced structure is \emph{irreducible} over $\mathbb{F}$ if all of its component subproblems are irreducible over $\mathbb{F}$.
\end{remark}

\subsection{Row Separation Problem}\label{sec:row separation problem}
The row separation problem introduced in this section originates from~\cite{Possible}.
In contrast to~\cite{Possible}, we define the row separation problem for an arbitrary set of row blocks $D_1,\ldots,D_{k+1}$, and then establish its connection with the decomposition of core problems
(\Cref{def:row separation problem constructed by core problem}). This formulation simplifies the notation in \Cref{sec:comparision basis} and makes it possible to compare different problems directly.

Furthermore, in~\cite{Possible} the row separation problem and the
decomposition of core problems are defined and studied over $\mathbb{R}$; here we extend them to $\mathbb{F}=\mathbb{C}$. Most arguments in~\cite{Possible} rely only on linear subspaces and their inclusion relations, and carry over directly to $\mathbb{F}=\mathbb{C}$ with the inner product fixed in
\Cref{sec:notations}. Where conjugation enters, namely in the construction of minimized subspaces in~\cite[Thm.~3.11]{Possible}, we state and prove the complex version in the appendix (\Cref{thm:extend of existence of minimized}). The splitting results of~\cite[Sec.~4]{Possible},
which rely on the singular value decomposition, are extended and discussed in \Cref{sec:division problem theory}.

\begin{definition}[Row Separation Problem in $\mathbb{F}$, \cite{Possible} for $\mathbb{F}=\mathbb{R}$]\label{def:row separation problem}
Let $D_1 \in \mathbb{F}^{r_1 \times \hat d}, \ldots, D_{k+1} \in \mathbb{F}^{r_{k+1} \times \hat d}$ be $k+1$ row blocks with a common number of columns. 
The \emph{row separation problem} associated with the block set $\{D_1,\ldots,D_{k+1}\}$ consists of finding:
\begin{itemize}
\item a permutation matrix $E$ of dimension $\sum_{j=1}^{k+1} r_j$,
\item unitary (orthogonal when $\mathbb{F}=\mathbb{R}$) matrices $H_1,\ldots,H_{k+1}$ of dimensions $r_1,\ldots,r_{k+1}$, respectively,
\item a unitary (orthogonal when $\mathbb{F}=\mathbb{R}$) matrix $R$ of dimension $\hat{d}$,
\item and a block diagonal matrix ${\rm diag}\{B^{(1)},\ldots,B^{(g)}\}$, where $g$ is an a priori unknown integer,
\end{itemize}
such that
\[
E\,{\rm diag}\{H_1,\ldots,H_{k+1}\}
\begin{bmatrix}
D_1\\ D_2\\ \vdots\\ D_{k+1}
\end{bmatrix}
R
=
{\rm diag}\{B^{(1)}, \ldots, B^{(g)}\}.
\]
\end{definition} 

\medskip
\noindent
Informally, the row separation problem seeks unitary (orthogonal when $\mathbb{F}=\mathbb{R}$) transformations that reorganize the row information carried by the blocks $D_1,\ldots,D_{k+1}$ into mutually independent structural components, thereby exposing decomposable structure that is not apparent in the original block formulation.

\begin{remark}
The matrices $D_j$ are not required to have full column rank.
In the rank-deficient case, the row separation problem admits trivial components corresponding to zero columns; see \cite[Thm.~4.4]{Possible}. These trivial components can be separated from the nontrivial ones and subsequently recombined, allowing the complete row separation structure to be recovered without loss of generality. We assume in the following discussion that each row block $D_j$ has full row rank. 
\end{remark}

In \cite{Possible}, the identification of component subproblems of a TLS core problem is reduced to the analysis of a collection of \emph{row separation subspaces}, which encode the essential decomposable structure of the data. 
\begin{definition}[Row Separation Subspace \cite{Possible}]
\label{def:row separation subspace}
A nonzero subspace $\mathcal{S}$ is called a \emph{row separation subspace} of the row separation problem associated with the block set $\{D_1,\ldots,D_{k+1}\}$ if there exists a solution of the row separation problem, with unitary matrix $R$ and block diagonal matrix ${\rm diag}\{B^{(1)},\ldots,B^{(g)}\}$ as in \Cref{def:row separation problem}, and a diagonal block $B^{(t)}$ such that
\[
\mathcal{S} = \mathcal{R}\!\left(\overline{R}\,(B^{(t)^{'}})^{T}\right),
\]
where $B^{(t)^{'}}$ denotes the full rows of $B^{(t)}$ in ${\rm diag}\{B^{(1)}, \ldots, B^{(g)}\}$.

\end{definition}
\begin{definition}[Row Separation Problem Constructed from a Core Problem]
\label{def:row separation problem constructed by core problem}
Let $A_{11} X_{11} \approx B_1$ be a core problem as described in \Cref{def:corepro}, and let
$U_1,\ldots,U_{k+1}$ denote orthonormal bases of the left singular vector subspaces associated with $A_{11}$, including a basis of $\mathcal{N}(A_{11}^*)$.
The row separation problem associated with the block set 
\[
\{U_1^* B_1,\ \ldots,\ U_{k+1}^* B_1\}
\]
is called the \emph{row separation problem constructed from the core problem}.
\end{definition}

\begin{remark}\label{re:possible structure}
It is shown in \cite{Possible} that any reduced structure of a real TLS core problem, obtained via orthogonal transformations, can be recovered from a suitable solution of the corresponding row separation problem constructed over $\mathbb{R}$. 
This correspondence uses only the uniqueness of singular subspaces and remains valid for unitary transformations over $\mathbb{C}$. Hence, after replacing orthogonal matrices by unitary ones and transposes by conjugate transposes in~\cite[Equation (3.3)]{Possible}, the correspondence continues to hold for $\mathbb{F}=\mathbb{C}$.

\end{remark}
\begin{remark}
By the properties of core problems~\cite{BGGK,TLSformulation}, the blocks $U_1^* B_1,\ldots,U_{k+1}^* B_1$ all have full row rank, and $\left[U_1 ,\ldots,U_{k+1}\right]^*B_1$ has full column rank.
\end{remark}

To facilitate the analysis of the full row separation problem, it is useful to consider lower-dimensional subproblems obtained by restricting attention to subsets of the row blocks.

\begin{definition}[$\{C\}$-Subset Subproblem \cite{Possible}]
\label{def:C-subset sub-problem}
Consider the row separation problem associated with the block set $\{D_1,\ldots,D_{k+1}\}$. 
Let $\{C\}=\{j_1,\ \cdots,\ j_c\} $ be an index subset such that $\{C\}\subset \{1,\ldots,k+1\}$. 
The \emph{$\{C\}$-subset subproblem} is defined as the row separation problem associated with the reduced row block set
$
\{D_{j_1},\ldots,D_{j_c}\}.
$
\end{definition}

\begin{remark}
To distinguish row separation subspaces associated with a $\{C\}$-subset subproblem from those of the full problem, we indicate the index set explicitly by a superscript, writing, for example, $\mathcal{S}^{\{C\}}$.
\end{remark}

To simplify the analysis of the collection of row separation subspaces $\{\mathcal{S}\}$, two notions describing minimal elements at different levels were introduced in \cite{Possible}. 
Minimized subspaces provide a canonical way to associate each direction vector with the smallest row separation subspace that contains it, while indivisible subspaces represent the atomic building blocks of the overall row separation structure.

\begin{definition}[Minimized Subspace \cite{Possible}]
\label{def:minimized subspace}
Consider the row separation problem associated with the block set $\{D_1,\ldots,D_{k+1}\}$ and a nonzero vector $v \in \mathcal{R}(D^{T})=\sum_{j}\mathcal{R}(D_j^T)$, where $D$ is the matrix obtained by stacking $D_1,\ldots,D_{k+1}$. 
The \emph{minimized subspace} $\mathcal{M}(v)$ is defined as a subspace satisfying:
\begin{enumerate}
\item $v \in \mathcal{M}(v)$;
\item $\mathcal{M}(v)$ is a row separation subspace;
\item for every row separation subspace $\mathcal{S}$ such that $v \in \mathcal{S}$, one has $\mathcal{M}(v) \subseteq \mathcal{S}$.
\end{enumerate}
\end{definition}

\begin{remark}\label{rem:minimized}
By condition~3 of \Cref{def:minimized subspace}, the minimized subspace
$\mathcal{M}(v)$ is unique whenever it exists. For real cases,
its existence and a construction are given in~\cite[Thm.~3.11]{Possible};
the corresponding result for complex cases is \Cref{thm:extend of existence of minimized} in the appendix.
\end{remark}

\begin{definition}[Indivisible Subspace \cite{Possible}]
\label{def:indivisible subspace}
A subspace $\mathcal{I}$ is called an \emph{indivisible subspace} of the row separation problem associated with the block set 
$\{D_1,\ldots,D_{k+1}\}$ if it satisfies:
\begin{enumerate}
\item $\mathcal{I} \neq \{0\}$;
\item $\mathcal{I}$ is a row separation subspace;
\item If there is a row separation subspace $\mathcal{S}$ satisfying $\mathcal{S} \subset \mathcal{I}$, then $\mathcal{S} = \mathcal{I}$.
\end{enumerate}
\end{definition}
\begin{remark}
Indivisible subspaces are further classified as \emph{unambiguous}, if they are orthogonal to all other indivisible subspaces, and \emph{ambiguous} otherwise; see~\cite[Secs.~3.2.1 and~3.2.2]{Possible}.
\end{remark}

The following example shows the difference between minimized subspace and indivisible subspace.
\begin{example}\label{eg:min ind}
Consider a row separation problem with
 \[D_1=\begin{bmatrix}
     1&\\&1
 \end{bmatrix},\qquad D_2=\begin{bmatrix}
     1&-2\\1&2
 \end{bmatrix}. 
 \]   
Let $v_1=\left[0\ 1\right]^T, v_2=\left[1\ -2\right]^T$. Then 
$\mathcal{M}\left(v_1\right)$ is an indivisible subspace, whereas $\mathcal{M}\left(v_2\right)$ is not.
\end{example}

The following characterization follows directly from
\Cref{def:minimized subspace} and \Cref{def:indivisible subspace} and provides an operational interpretation of indivisibility.
\begin{lemma}[\cite{Possible}]
\label{lem:indivisible minimized subspace}
A row separation subspace $\mathcal{I}$ is an indivisible subspace if and only if, for every $v \in \mathcal{I}$, one has
$
\mathcal{M}(v) = \mathcal{I}.
$
\end{lemma}

\begin{remark}\Cref{lem:indivisible minimized subspace} recasts the search for indivisible subspaces as the identification of special direction vectors $v$ whose associated minimized subspaces are indivisible.
\end{remark}

\section{Row Separation Problem Supplement}\label{sec:supplement}

This section collects tools used in the later sections. The results use elementary invariant-subspace properties of the row-space Gram operators $D_j^T\bar D_j$, together with inclusion and orthogonality relations among subspaces, with orthogonality taken with respect to the inner product of \Cref{sec:notations}. They hold for row separation problems over both $\mathbb{R}$ and $\mathbb{C}$.
In terms of the overview in \Cref{sec:roadmap}, \Cref{le:D_j and S cap} provides the invariant-subspace characterization on which the rest of the paper builds. The remaining results of \Cref{sec:structure property} describe how indivisible subspaces generate all row separation subspaces, and \Cref{sec:comparision basis} supplies the subset exchange (\Cref{thm:subset exchange}), which allows the local replacements of \Cref{sec:division problem construction in complex number}.

\subsection{Structure Property}\label{sec:structure property}

We begin with a lemma that characterizes row separation subspaces through the row-space Gram operators of the blocks and describes how a row separation subspace meets the row space of a single block.

\begin{lemma}\label{le:D_j and S cap}
Consider the row separation problem associated with the block set $\{D_1,\ldots,D_{k+1}\}$, whose blocks have full row rank, and let $K_j=D_j^T\bar D_j$ denote the row-space Gram operators, $j=1,\ldots,k+1$.
\begin{enumerate}
\item If $\mathcal{S}$ is a row separation subspace with orthonormal basis $S$, then $K_j\mathcal{S}\subseteq\mathcal{S}$ and
\[
\mathcal{S}\cap \mathcal{R}\left(D_j^T\right)=\mathcal{R}\left(D_j^T\bar{D}_jS\right),\qquad j=1,\ldots,k+1.
\]
\item Conversely, every nonzero subspace $\mathcal{S}\subseteq\sum_{j}\mathcal{R}\left(D_j^T\right)$ with $K_j\mathcal{S}\subseteq\mathcal{S}$ for $j=1,\ldots,k+1$ is a row separation subspace.
\end{enumerate}
\end{lemma}
\begin{proof}
For (1), consider a solution of the row separation problem (\Cref{def:row separation problem}) with matrices $E$, $H_1,\ldots,H_{k+1}$, $R$, in which $\mathcal{S}$ corresponds to the diagonal block $B^{(t)}$ (\Cref{def:row separation subspace}), and write $R=\bar W$ with $W$ unitary. For a row $x^T$ of some $H_jD_j$ we have $x^TR=(W^*x)^T$, so the coordinates of $x$ in the basis $W$ are supported in the columns of the diagonal block that contains that row after the permutation $E$. Since the column supports of different diagonal blocks are disjoint, every row of every $H_jD_j$, written as a column vector $x$, lies either in $\mathcal{S}$ (if it belongs to $B^{(t)}$) or in $\mathcal{S}^{\perp}$. As $H_j^T\bar H_j=I$,
\[
K_j=(H_jD_j)^T\,\overline{(H_jD_j)}=\sum_{p}x_{j,p}x_{j,p}^*,
\]
where $x_{j,p}$ runs over the rows of $H_jD_j$. Hence $K_j\mathcal{S}\subseteq{\rm span}\{x_{j,p}\in\mathcal{S}\}\subseteq\mathcal{S}$. Since $K_j$ is Hermitian, also $K_j\mathcal{S}^{\perp}\subseteq\mathcal{S}^{\perp}$, and as $\mathcal{R}(D_j^T)=\mathcal{R}(K_j)=K_j\mathcal{S}+K_j\mathcal{S}^{\perp}$,
\[
\mathcal{R}(D_j^T)=K_j\mathcal{S}\oplus K_j\mathcal{S}^{\perp},\qquad K_j\mathcal{S}\subseteq\mathcal{S},\quad K_j\mathcal{S}^{\perp}\subseteq\mathcal{S}^{\perp}.
\]
If $x\in\mathcal{S}\cap\mathcal{R}(D_j^T)$, write $x=a+b$ with $a\in K_j\mathcal{S}$ and $b\in K_j\mathcal{S}^{\perp}$; then $b=x-a\in\mathcal{S}\cap\mathcal{S}^{\perp}=\{0\}$. Therefore $\mathcal{S}\cap\mathcal{R}(D_j^T)=K_j\mathcal{S}=\mathcal{R}(D_j^T\bar D_jS)$.

For (2), let $\{0\}\neq\mathcal{S}\subseteq\sum_j\mathcal{R}(D_j^T)$ with $K_j\mathcal{S}\subseteq\mathcal{S}$ for all $j$. As in the proof of (1), $K_j\mathcal{S}^{\perp}\subseteq\mathcal{S}^{\perp}$ and
\[
\mathcal{R}(D_j^T)=\big(\mathcal{S}\cap\mathcal{R}(D_j^T)\big)\oplus\big(\mathcal{S}^{\perp}\cap\mathcal{R}(D_j^T)\big).
\]
Fix $j$ and let $M_j=\bar D_jD_j^T$, which is Hermitian positive definite since $D_j$ has full row rank. Let $\mathcal{Y}_j=\{y:\ D_j^Ty\in\mathcal{S}\}$ and $\mathcal{Y}_j'=\{y:\ D_j^Ty\in\mathcal{S}^{\perp}\}$. Since $D_j^T$ is injective, $\mathcal{Y}_j\oplus\mathcal{Y}_j'=\mathbb{F}^{r_j}$. For $y\in\mathcal{Y}_j$ we have $D_j^TM_jy=K_jD_j^Ty\in\mathcal{S}$, so $M_j\mathcal{Y}_j\subseteq\mathcal{Y}_j$, and since $M_j$ is nonsingular, $M_j\mathcal{Y}_j=\mathcal{Y}_j$. For $y\in\mathcal{Y}_j$ and $z\in\mathcal{Y}_j'$, $0=\langle D_j^Tz,D_j^Ty\rangle=y^*M_jz=(M_jy)^*z$; as $M_j\mathcal{Y}_j=\mathcal{Y}_j$, this gives $\mathcal{Y}_j'\subseteq\mathcal{Y}_j^{\perp}$, and by dimension $\mathcal{Y}_j'=\mathcal{Y}_j^{\perp}$. Let $H_j$ be the unitary matrix whose rows are $h^T$ for the vectors $h$ of an orthonormal basis of $\mathcal{Y}_j$ followed by one of $\mathcal{Y}_j^{\perp}$. Each row of $H_jD_j$, written as a column vector $D_j^Th$, then lies in $\mathcal{S}$ or in $\mathcal{S}^{\perp}$.

Let $W=[W_1,W_2,W_0]$ be unitary, where the columns of $W_1$, $W_2$ and $W_0$ form orthonormal bases of $\mathcal{S}$, of $\mathcal{S}^{\perp}\cap\sum_j\mathcal{R}(D_j^T)$ and of $\big(\sum_j\mathcal{R}(D_j^T)\big)^{\perp}$, respectively, and put $R=\bar W$. For a row $x^T$ of some $H_jD_j$, the vector $x^TR=(W^*x)^T$ is supported in the columns of $W_1$ if $x\in\mathcal{S}$ and in those of $W_2$ otherwise; the columns of $W_0$ are zero for every row and can be appended to either column group. Grouping the rows of all blocks according to whether they lie in $\mathcal{S}$ therefore gives a solution of the row separation problem in which $\mathcal{S}$ is the row separation subspace of the first diagonal block.
\end{proof}

\begin{remark}
\Cref{le:D_j and S cap} characterizes the row separation subspaces as the nonzero subspaces of $\sum_j\mathcal{R}(D_j^T)$ that are invariant under all row-space Gram operators $D_j^T\bar D_j$. Invariance under the orthogonal projectors onto the $\mathcal{R}(D_j^T)$ alone is not sufficient: for a single block, a one-dimensional subspace of $\mathcal{R}(D_1^T)$ is a row separation subspace only if it is spanned by an eigenvector of $D_1^T\bar D_1$ (compare \Cref{cor:1d equivalent}). The two notions coincide when $\bar D_jD_j^T$ is a scalar matrix, as for the blocks of a $1$-division problem.
\end{remark}

\begin{corollary}\label{cor:subspace with subset}
Consider the row separation problem associated with the block set $\{D_1,\ldots,D_{k+1}\}$, let $\mathcal{S}$ denote a row separation subspace and let $\{C\}=\{j_1,\ldots,j_c\}\subset\{1,\ldots,k+1\}$ denote a non-repetitive index subset. Let 
$$\mathcal{S}^{\{C\}}=\sum_{j\in\{C\}}\big(\mathcal{S}\cap\mathcal{R}\left(D_j^T\right)\big).$$ 
If $\mathcal{S}^{\{C\}}\neq \{0\}$, then $\mathcal{S}^{\{C\}}$ is a row separation subspace of the $\{C\}$-subset subproblem.
\end{corollary}
\begin{proof}
Let $j\in\{C\}$. By \Cref{le:D_j and S cap}(1), $D_j^T\bar D_j\mathcal{S}^{\{C\}}\subseteq D_j^T\bar D_j\mathcal{S}=\mathcal{S}\cap\mathcal{R}(D_j^T)\subseteq\mathcal{S}^{\{C\}}$. Hence $\mathcal{S}^{\{C\}}$ is a nonzero subspace of $\sum_{j\in\{C\}}\mathcal{R}(D_j^T)$ invariant under the Gram operators of the $\{C\}$-subset subproblem, and \Cref{le:D_j and S cap}(2) applies.
\end{proof}

We now relate the minimized subspaces of a $\{C\}$-subset subproblem to those of the original row separation problem. 
\begin{lemma}\label{cor:subset minimized}
Consider the row separation problem associated with the block set $\{D_1,\ldots,D_{k+1}\}$ and a non-repetitive index subset $\{C\}=\{j_1,\ldots,j_c\}\subset\{1,\ldots,k+1\}$. 
Let $\mathcal{M}^{\{C\}}(v)$ denote the minimized subspace of a vector $v\in\mathcal{R}\left(D_j^T\right)$, $j\in\{C\}$, in the $\{C\}$-subset subproblem, and let $\mathcal{M}(v)$ denote the minimized subspace of $v$ in the original row separation problem. 
Then, for any $v\in \mathcal{R}\!\left(D_{j_i}^T\right)$ with $j_i\in \{C\}$, one has
\[
\mathcal{M}^{\{C\}}(v)\subseteq \mathcal{M}(v).
\]
\end{lemma}
\begin{proof}
Let 
$$ \mathcal{S}^{\{C\}}=\sum_{j\in\{C\}}\big(\mathcal{M}(v)\cap\mathcal{R}\left(D_j^T\right)\big),$$
which contains $v$. By \Cref{cor:subspace with subset}, $\mathcal{S}^{\{C\}}$ is a row separation subspace of the $\{C\}$-subset subproblem containing $v$. Therefore $\mathcal{M}^{\{C\}}(v)\subseteq \mathcal{S}^{\{C\}}\subseteq \mathcal{M}(v)$.
\end{proof}

The following two lemmas show that indivisible subspaces provide a complete description of the row separation structure; they are used in \Cref{sec:comparision basis} and \Cref{sec:alternative construction}.

\begin{lemma}\label{thm:indivisible to row separation subspace}
Consider the row separation problem associated with the block set $\{D_1,\ldots,D_{k+1}\}$ and let
$\mathcal{S}$ be a row separation subspace.
Then there exist indivisible subspaces $\mathcal{I}_1,\ldots,\mathcal{I}_g$ such that
\[
\mathcal{S} = \bigoplus_{t=1}^{g} \mathcal{I}_t,
\]
with $\mathcal{I}_t \perp \mathcal{I}_{t'}$ for all $t \neq t'$.
\end{lemma}

\begin{proof}
We use induction on $\dim\mathcal{S}$. If $\mathcal{S}$ is not indivisible, there is a row separation subspace $\mathcal{S}'$ with $\{0\}\neq\mathcal{S}'\subsetneq\mathcal{S}$. Let $\mathcal{T}=\mathcal{S}\cap(\mathcal{S}')^{\perp}$, so that $\mathcal{S}=\mathcal{S}'\oplus\mathcal{T}$ and $\mathcal{T}\neq\{0\}$. By \Cref{le:D_j and S cap}(1), $\mathcal{S}$ and $\mathcal{S}'$ are invariant under the Hermitian Gram operators $D_j^T\bar D_j$; hence so is $(\mathcal{S}')^{\perp}$, and therefore $\mathcal{T}$, which is a row separation subspace by \Cref{le:D_j and S cap}(2). Applying the induction hypothesis to $\mathcal{S}'$ and $\mathcal{T}$, both of dimension smaller than $\dim\mathcal{S}$, gives the result.
\end{proof}

\begin{remark}
The choice of the $\mathcal{I}_t$ need not be unique; see~\cite[Sec.~3.2.2]{Possible}. Together with \Cref{cor:row separation subspace sum}, \Cref{thm:indivisible to row separation subspace} shows that the family of row separation subspaces is entirely determined by the collection of indivisible subspaces.
\end{remark}

\begin{lemma}\label{cor:row separation subspace sum}
Consider the row separation problem associated with the block set $\{D_1,\ldots,D_{k+1}\}$ and two
row separation subspaces $\mathcal{S}$ and $\mathcal{S}'$.
If $\mathcal{S} \perp \mathcal{S}'$, then there exists a row separation subspace $\mathcal{S}''$
such that
\[
\mathcal{S}'' = \mathcal{S} \oplus \mathcal{S}'.
\]
\end{lemma}
\begin{proof}
By \Cref{le:D_j and S cap}(1), $\mathcal{S}$ and $\mathcal{S}'$ are invariant under all Gram operators $D_j^T\bar D_j$, hence so is $\mathcal{S}\oplus\mathcal{S}'$; by \Cref{le:D_j and S cap}(2), it is a row separation subspace.
\end{proof}

\subsection{Comparison Property}\label{sec:comparision basis}
While~\cite{Possible} focuses on the internal structure of a single row separation problem, this subsection studies relations between different row separation problems. These relations depend only on the row blocks themselves, not on where the blocks come from. Hence they are not restricted to the strategy of \Cref{sec:division problem construction in complex number}, nor to TLS core
problems. Via \Cref{def:row separation problem constructed by core problem}, they also allow core problems to be compared with one another.

We introduce three relations, which answer the following questions:
\begin{enumerate}
\item \emph{Equivalence}: Different collections of
row blocks may describe the same decomposition, e.g., when a block is
duplicated. When should two row separation problems be regarded as the same, and how can this be checked? (\Cref{def:equivalent with row separation subspaces}
and \Cref{thm:same indivisible same row separation})
\item \emph{Preorder}: Is there a weaker relation than equivalence between different row separation problems? (\Cref{def:dense problem})
\item \emph{Composition and exchange}: How can a row separation problem equivalent to a given one be constructed? (\Cref{def:row block composition} and \Cref{thm:subset exchange})
\end{enumerate}

Throughout this subsection and the remainder of the paper, the symbol $\mathsf{P}$ will be used to denote a generic row separation problem.

We begin by introducing a notion of equivalence that isolates the intrinsic structure of a row separation problem and disregards inessential differences in its representation.
Since the decomposition properties of interest depend only on the family of row separation subspaces generated by a problem, it is natural to regard two problems as equivalent whenever they induce the same collection of such subspaces.

\begin{definition}[Equivalent with respect to Row Separation Subspaces]
\label{def:equivalent with row separation subspaces}
Two row separation problems $\mathsf{P}_{1}$ and $\mathsf{P}_2$
are said to be 
\emph{equivalent with respect to row separation subspaces} if they have the same collection of row separation subspaces.
\end{definition}

\begin{remark}
This notion of equivalence abstracts away from the specific representation of the row blocks and retains only the induced row separation structure.
More generally, this equivalence allows one to replace a given row separation problem with an alternative but structurally identical formulation—possibly simpler or more convenient—while preserving all information relevant to decomposition, which is also the main idea of the strategy to be proposed in \Cref{sec:division problem construction in complex number}.
\end{remark}

\begin{remark}
For example, the row separation problems associated with 
$\{D_1,\ldots,D_{k+1}\}$ and $\{D_1,\ldots,D_{k+1},D_1\}$ are equivalent in this sense, since duplicating a row block does not generate new row separation subspaces. Moreover, Example~\ref{eg:min ind} is equivalent to the row separation problem associated with $\{\left[1,\ 0\right],\ \left[0,\ 1\right]\}$.   
\end{remark}

The next lemma simplifies the equivalence condition and
provides a strong justification for focusing on indivisible subspaces in the analysis and comparison of different row separation problems.

\begin{lemma}\label{thm:same indivisible same row separation}

Consider two row separation problems $\mathsf{P}_{1}$ and $\mathsf{P}_2$.
If $\mathsf{P}_1$ and $\mathsf{P}_2$ have the same collection of indivisible subspaces, then they are equivalent with respect to row separation subspaces.
\end{lemma}
\begin{proof}
This is a direct consequence of \Cref{thm:indivisible to row separation subspace} and \Cref{cor:row separation subspace sum}.
\end{proof}

We now introduce a relation that allows row separation problems to be compared in terms of the amount of structural information they encode.

This relation defines a preorder on row separation problems and hence a partial order on their equivalence classes. Via Definition 2.9, it also provides a corresponding comparison of core problems based on the inclusion of their row separation subspaces.

\begin{definition}[Dense Relationship]\label{def:dense problem}
Let $\mathsf{P}_1$ and $\mathsf{P}_2$ be two row separation problems.
We say that $\mathsf{P}_1$ is \emph{denser} than $\mathsf{P}_2$ if every row separation subspace of $\mathsf{P}_2$ is also a row separation subspace of $\mathsf{P}_1$.
\end{definition}

\begin{remark}
If $\mathsf{P}_1$ is denser than $\mathsf{P}_2$ and $\mathsf{P}_2$ is denser than $\mathsf{P}_1$, then the two problems admit exactly the same collection of row separation subspaces.
More generally, the dense relationship allows one to compare non-equivalent problems and to formalize when one problem strictly refines the row separation structure of another.
\end{remark}

\begin{remark}
A row separation problem with real row blocks can be considered over $\mathbb{R}$ or over $\mathbb{C}$. Since every real solution is also a complex one, the complex span of every real row separation subspace is a row separation subspace over $\mathbb{C}$; in this sense the problem over $\mathbb{C}$ is denser than the problem over $\mathbb{R}$.
\end{remark}

We next introduce a construction that allows one to assemble a global row separation problem from several smaller ones, while preserving their local row separation structures.
This operation plays a key conceptual role in the sequel: it provides a formal mechanism for reconstructing a full problem from independently analyzed subproblems.

\begin{definition}[Row Block Composition]\label{def:row block composition}
Let $\{\mathsf{P}_1,\ldots,\mathsf{P}_l\}$ be a collection of row separation problems whose row blocks all have the same number of columns, where $\mathsf{P}_s$ has $k_s+1$ row blocks for $s=1,\ldots,l$.
The \emph{row block composition} of these problems is the row separation problem $\mathsf{P}^{\mathrm{Com}}$ associated with the union of their block sets; it has $\sum_{s=1}^{l}k_s+l$ row blocks, and for each $s$ its subset subproblem formed by the blocks of $\mathsf{P}_s$ is $\mathsf{P}_s$ itself.
\end{definition}

We conclude this section with a result that synthesizes the comparison framework developed above and makes precise when subset subproblems can be replaced without altering the global row separation structure.
The result provides the theoretical justification for modifying or reconstructing parts of a row separation problem while maintaining structural equivalence, and will serve as a key tool in the alternative constructions developed in \Cref{sec:alternative construction}.

\begin{theorem}[Subset Exchange]\label{thm:subset exchange}
Let $\mathsf{P}_1$ be a row separation problem associated with the block set
$\{D_1,\ldots,D_{k+1}\}$, and let $\{C_1\}$ and $\{C_2\}$ be index subsets such that
\[
\{C_1\} \cup \{C_2\} = \{1,\ldots,k+1\},
\qquad
\{C_1\} \cap \{C_2\} = \emptyset.
\]
Let $\mathsf{P}_2$ be a row separation problem with full-row-rank blocks having the same number of columns as the $D_j$, that is equivalent with respect to row separation subspaces to the $\{C_1\}$-subset subproblem of $\mathsf{P}_1$. Let $\mathsf{P}_3$ be the row separation problem associated with the union of the block set of $\mathsf{P}_2$ and $\{D_j:\ j\in\{C_2\}\}$, that is, the row block composition of $\mathsf{P}_2$ and the $\{C_2\}$-subset subproblem of $\mathsf{P}_1$ (\Cref{def:row block composition}).
Then $\mathsf{P}_3$ and $\mathsf{P}_1$ are equivalent with respect to row separation subspaces.
\end{theorem}

\begin{proof}
For a row separation problem $\mathsf{P}$ with blocks $E_1,\ldots,E_n$, let $\mathcal{V}(\mathsf{P})=\sum_l\mathcal{R}(E_l^T)$; we call $E_l^T\bar E_l$ the Gram operators of $\mathsf{P}$. By \Cref{le:D_j and S cap}, the row separation subspaces of $\mathsf{P}$ are exactly the nonzero subspaces of $\mathcal{V}(\mathsf{P})$ that are invariant under all Gram operators of $\mathsf{P}$. Two consequences will be used.
(a) $\mathcal{V}(\mathsf{P})$ is the largest row separation subspace of $\mathsf{P}$.
(b) A subspace invariant under all Gram operators of $\mathsf{P}$ is invariant under the orthogonal projector onto $\mathcal{V}(\mathsf{P})$: the matrix $A=\sum_lE_l^T\bar E_l$ is Hermitian positive semidefinite with $\mathcal{R}(A)=\mathcal{V}(\mathsf{P})$, so this projector equals $p(A)$ for a polynomial $p$ with $p(0)=0$ and $p(\lambda)=1$ at the nonzero eigenvalues of $A$.

Let $\mathsf{Q}$ denote the $\{C_1\}$-subset subproblem of $\mathsf{P}_1$, let $\mathcal{V}_1=\mathcal{V}(\mathsf{Q})=\sum_{j\in\{C_1\}}\mathcal{R}(D_j^T)$, and let $\Pi$ be the orthogonal projector onto $\mathcal{V}_1$. Since $\mathsf{P}_2$ and $\mathsf{Q}$ have the same row separation subspaces, (a) gives $\mathcal{V}(\mathsf{P}_2)=\mathcal{V}_1$; hence $\mathcal{V}(\mathsf{P}_3)=\mathcal{V}(\mathsf{P}_1)$.

We claim: if a subspace $\mathcal{S}$ is invariant under all Gram operators of one of $\mathsf{Q}$, $\mathsf{P}_2$, then it is invariant under all Gram operators of the other. Indeed, by (b), $\Pi\mathcal{S}\subseteq\mathcal{S}$, hence $\Pi\mathcal{S}=\mathcal{S}\cap\mathcal{V}_1$. The subspace $\mathcal{S}\cap\mathcal{V}_1\subseteq\mathcal{V}_1$ is invariant under the Gram operators of the first problem, because both $\mathcal{S}$ and $\mathcal{V}_1$ are, so it is either zero or a row separation subspace of that problem, hence of the other one; in either case it is invariant under every Gram operator $L$ of the other problem. Such an $L$ is Hermitian with $\mathcal{R}(L)\subseteq\mathcal{V}_1$, so $\mathcal{V}_1^{\perp}\subseteq\mathcal{N}(L)$, that is, $L=L\Pi$. Therefore
\[
L\mathcal{S}=L\Pi\mathcal{S}=L\big(\mathcal{S}\cap\mathcal{V}_1\big)\subseteq\mathcal{S}\cap\mathcal{V}_1\subseteq\mathcal{S}.
\]

The Gram operators of $\mathsf{P}_1$ are those of $\mathsf{Q}$ together with $D_j^T\bar D_j$, $j\in\{C_2\}$, and those of $\mathsf{P}_3$ are those of $\mathsf{P}_2$ together with the same $D_j^T\bar D_j$, $j\in\{C_2\}$. By the claim, a nonzero subspace of $\mathcal{V}(\mathsf{P}_1)=\mathcal{V}(\mathsf{P}_3)$ is invariant under all Gram operators of $\mathsf{P}_1$ if and only if it is invariant under all Gram operators of $\mathsf{P}_3$. By \Cref{le:D_j and S cap}, $\mathsf{P}_1$ and $\mathsf{P}_3$ have the same row separation subspaces.
\end{proof}

\begin{remark}
By \Cref{thm:same indivisible same row separation}, the hypothesis on $\mathsf{P}_2$ holds in particular when $\mathsf{P}_2$ and the $\{C_1\}$-subset subproblem of $\mathsf{P}_1$ have the same indivisible subspaces.
\end{remark}

In summary, the row separation structure of a problem is determined by its indivisible subspaces, and parts of a problem can be replaced by equivalent ones without changing it. \Cref{sec:division problem theory} identifies the problems in which the indivisible subspaces can be read off directly, and the obstruction that prevents this in general.

\section{Division Problem Theory}\label{sec:division problem theory}
\Cref{sec:structure property} reduced the search for components to the search for indivisible subspaces. To find them, we study in this section a special class of row separation problems, called division problems. We first recall the $1$- and $2$-division problems of~\cite{Possible}, defined there over $\mathbb{R}$, and then extend them to $c$-division problems for $c\ge3$, over both $\mathbb{R}$ and $\mathbb{C}$.
The parameter $c$ measures the level of homogeneity of the row blocks; at the top level, a complete division problem is fully homogeneous: within each block, no direction is distinguished from the others by its interaction with the remaining blocks. 
Complete division problems are the main termination condition of the strategy in \Cref{sec:division problem construction in complex number}.
The central results of this section are \Cref{thm:relation between eigen/singular}, which characterizes complete division problems by scalar cycle products, and \Cref{thm:single cycle}, which shows that, when homogeneity is tested on cycles of increasing length, the first obstruction lies on a chordless cycle. \Cref{sec:$3$-division} recalls the low-order cases, and \Cref{sec:c-division} develops the general theory.

All formulas in this section are written over $\mathbb{C}$ with the inner product of \Cref{sec:notations}; for real cases the conjugation acts trivially, and they reduce to the real formulas without ambiguity.

\subsection{Division Problem of $c\leq 2$}\label{sec:$3$-division}
The $1$- and $2$-division problems below were introduced in~\cite{Possible}
over $\mathbb{R}$. Their definitions and the results we use carry over to
$\mathbb{F}=\mathbb{C}$ by using the complex inner product of \Cref{sec:notations}. The $3$-division problem is also defined in~\cite{Possible}, based on the 
eigenvalue property of a real covariance matrix $D_1D_2^TD_2D_3^TD_3D_1^T$. In this definition, the covariance matrix is a multiple of an orthogonal matrix and may have complex conjugate eigenvalues; splitting a $3$-division
problem along individual eigenspaces, or replacing it by an equivalent problem, then involves complex eigenvectors. For similar reasons, we define $3$-division problems through the
subspace-based definition of $c$-division problems in \Cref{def:c-division}, which applies over $\mathbb{R}$ and $\mathbb{C}$ alike. By
\Cref{thm:relation between eigen/singular}, it agrees with the definition from~\cite{Possible} whenever the latter applies.

\begin{definition}[$1$-division Problem in $\mathbb{F}$ \cite{Possible}]\label{def:1d}
The row separation problem associated with the block set $\{D_1,\ \cdots,\ D_{k+1}\}$ is called a $1$-division problem if the following conditions hold:
\begin{enumerate}
\item For all $j\in \{1,\ \cdots,\ k+1\}$, the matrix $\bar{D}_j D_j^T=\alpha_j I$, where $\alpha_j\neq 0$.
\item For all $j\in \{1,\ \cdots,\ k+1\}$, $\operatorname{rank}(D_j)\geq 2$.
\end{enumerate}
\end{definition}

\begin{definition}[$2$-division Problem in $\mathbb{F}$ \cite{Possible}]
\label{def:2d}
The row separation problem associated with the row block set
$\{D_1,\ \cdots,\ D_{k+1}\}$ is called a $2$-division problem if the following
conditions hold:
\begin{enumerate}
\item It is a $1$-division problem.

\item For all $j,j' = 1,\ \cdots,\ k+1$, the matrix
$\bar{D}_j D_{j'}^{T}$ is a scalar multiple of a unitary matrix. 

\item For any indices $j,j' \in \{1,\ldots,k+1\}$, there exists
a finite sequence of indices $(\ell_0,\ell_1,\ldots,\ell_m)$ such that
\[
\ell_0 = j,\quad \ell_m = j',
\qquad\text{and}\qquad
\bar{D}_{\ell_t}D^T_{\ell_{t+1}}\neq 0
\quad \text{for all } t = 0,\ldots,m-1.
\]
\end{enumerate}
\end{definition}

\begin{remark}
The definition of $2$-division extends the notion of division by treating each row block as a coherent unit rather than as isolated rows. In particular, in a $2$-division problem all $D_j$ have the same number of rows, $r_1=\cdots=r_{k+1}$.    
\end{remark}

\begin{lemma}\label{le:after 2-division including}
Consider a $2$-division problem. Then for every $j$, every nonzero $v\in\mathcal{R}\left(D_j^T\right)$ and every $i$ with $\bar{D}_iD_j^T\neq 0$, we have $D_i^T\bar{D}_iv\in\mathcal{M}\left(v\right)$
and $\mathcal{M}\left(v\right)=\mathcal{M}\left(D_i^T\bar{D}_iv\right)$.

\end{lemma}
\begin{proof}
This is a direct consequence of \Cref{le:D_j and S cap} and \Cref{def:minimized subspace}.
\end{proof}

\begin{remark}

Thus equal minimized subspaces propagate along the chains in Condition~3 of \Cref{def:2d}.
\end{remark}
The $1$- and $2$-division conditions impose strong constraints on the row blocks: each block has mutually orthogonal rows of equal norm, every cross-product $\bar D_iD_j^T$ is a multiple of a unitary matrix, and any two blocks are linked by a chain of blocks with nonzero cross-products. These constraints are the foundation of the subsequent construction: the discussion in \Cref{sec:c-division} is based on division problems that inherit the $1,2$-division conditions.

\medskip

The singular value splitting of \cite[Sec.~4]{Possible} relies on the singular value decomposition and is stated there over $\mathbb{R}$. Combined with \Cref{thm:subset exchange}, the constructions of $1$- and $2$-division problems in \cite[Sec.~4.1 and~4.2]{Possible} carry over to $\mathbb{C}$. The two corollaries below give the equivalent constructions for a single block and for a pair of blocks. For a row separation problem with several row blocks, each block $D_j$ can be replaced in turn by the blocks $D_j^{1},\dots,D_j^{l_j}$ of \Cref{cor:1d equivalent} without changing the row separation subspaces of the whole problem, by \Cref{thm:subset exchange} applied with $\{C_1\}=\{j\}$. In the same way, for a pair of indices $i,j$, the blocks $D_i$ and $D_j$ can be replaced by the blocks of \Cref{cor:2d equivalent}, by \Cref{thm:subset exchange} applied with $\{C_1\}=\{i,j\}$. These replacements are Steps~1.1 and~2.2 of the strategy in \Cref{sec:final strategy}.

\begin{corollary}\label{cor:1d equivalent}
Consider the row separation problem $\mathsf{P}$ associated with a single row block $D_1$ of full row rank. Let $D_1=U\Sigma V^{*}$ be a compact singular value decomposition and $V=[V_1,\dots,V_l]$ partitioned according to the distinct singular values $\sigma_1>\dots>\sigma_l>0$. Then $\mathsf{P}$ is equivalent, with respect to row separation subspaces, to the row separation problem associated with $\{D_1^{1},\dots,D_1^{l}\}$, where
\[
D_1^{s}=\sigma_sV_s^{*},\qquad s=1,\dots,l .
\]
\end{corollary}

\begin{remark}
The row separation problem associated with a single row block $D^{s}_1$ is a $1$-division problem if and only if the multiplicity of $\sigma_s$ is at least $2$.

\end{remark}

\begin{corollary}\label{cor:2d equivalent}
Consider a $1$-division problem $\mathsf{P}$ associated with row blocks $D_1\in\mathbb{C}^{r_1\times\hat d}$ and $D_2\in\mathbb{C}^{r_2\times\hat d}$. Let $\bar D_1D_2^{T}=U\Sigma V^{*}$ be a compact singular value decomposition with unitary $U\in\mathbb{C}^{r_1\times r_1}$ and $V\in\mathbb{C}^{r_2\times r_2}$, and let $\sigma_1>\dots>\sigma_l>0$ be the distinct nonzero singular values. Partition $U=[U_1,\dots,U_l,U_0]$ and $V=[V_1,\dots,V_l,V_0]$, where the columns of $U_s$ and $V_s$ are orthonormal bases of left and right singular subspaces of $\sigma_s$, and the columns of $U_0$ and $V_0$ are orthonormal bases of $\mathcal{N}\big((\bar D_1D_2^{T})^{*}\big)$ and $\mathcal{N}(\bar D_1D_2^{T})$; either of $U_0$, $V_0$ may be empty.

Then $\mathsf{P}$ is equivalent, with respect to row separation subspaces, to the row separation problem associated with 
$$\{U_0^TD_1,\ U_1^TD_1,\ \dots,U_l^TD_1,\ V_0^TD_2,\ V_1^TD_2,\ \dots,V_l^TD_2 \}$$.
\end{corollary}

\begin{remark}
The problem obtained in \Cref{cor:2d equivalent} is in general not a $2$-division problem. Moreover, if some $U_s$ or $V_s$ has only one column, the corresponding block consists of a single row and the problem is not even a $1$-division problem. For this reason, the strategy in \Cref{sec:final strategy} returns to Step~1.2 after every execution of Step~2.2.
\end{remark}

\begin{remark}
 In both corollaries the replacement amounts to applying a unitary matrix to a block and grouping its rows: $H_1=U^{*}$ in \Cref{cor:1d equivalent}, and $H_1=U^{T}$, $H_2=V^{T}$ in \Cref{cor:2d equivalent}. Each new block may moreover be multiplied by a nonzero constant, for instance to normalize its rows for numerical stability; this does not change the row separation subspaces of the whole problem, since whether a row lies in a given subspace $\mathcal{S}$ or in $\mathcal{S}^{\perp}$ is unaffected by scaling.   
\end{remark}

\subsection{General $c$-Division Problems for $c \geq 3$}
\label{sec:c-division}

The $2$-division conditions only constrain pairs of row blocks. In this subsection we consider homogeneity among larger collections of blocks: a $c$-division problem requires that, for any $c$ row blocks, every direction lies in a row separation subspace meeting each of these blocks in a
one-dimensional subspace. However, this condition is formulated in terms of subspaces and \Cref{thm:relation between eigen/singular} turns it into a spectral condition: the problem is a complete division problem if and only if all
covariance matrices along orderings of the blocks are scalar matrices. Moreover, when a $c$-division problem fails to be a $(c+1)$-division problem, the inhomogeneity is concentrated on a chordless cycle (\Cref{thm:single cycle} and \Cref{thm:chordless}); the splitting in \Cref{sec:division problem construction in complex number} is carried out on such cycles.

\begin{definition}[$c$-division Problem in $\mathbb{F}$]
\label{def:c-division}
The row separation problem associated with the block set 
$\{D_1, \ \cdots, \ D_{k+1}\}$ is called a $c$-division problem ($c \geq 3$) if the following conditions are satisfied:
\begin{enumerate}
\item It is a $2$-division problem.
\item For any non-repetitive index collection $\{C\}=\{j_1,\ \cdots,\ j_c\}$, 
and for any 
$v\in\mathcal{R}\!\left(D^T_{j_1}\right)$, there exists a row separation subspace 
$\mathcal{S}^{\{C\}}$ of the $\{C\}$-subset subproblem such that 

\[
v\in\mathcal{S}^{\{C\}},
\qquad
\dim\!\left(\mathcal{S}^{\{C\}} \cap \mathcal{R}\!\left(D^T_{j_i}\right)\right)=1,
\qquad i=1,\ \cdots,\ c.
\]
\end{enumerate}
\end{definition}

\begin{remark}
For $3\le c\le k+1$, the condition in \Cref{def:c-division} for index subsets of cardinality $c$ implies the same condition for all smaller cardinalities, by restricting $\mathcal{S}^{\{C\}}$ to a smaller subset (see \Cref{le:D_j and S cap} and \Cref{cor:subspace with subset}). Hence every $c$-division problem is a $(c-1)$-division problem.

\end{remark}

\begin{definition}[Complete Division Problem in $\mathbb{F}$]
\label{def:complete division problem}
The row separation problem associated with the block set
$\{D_1, \ \cdots, \ D_{k+1}\}$ is called a complete division problem if it is a
$(k+1)$-division problem.
\end{definition}
\begin{example}
\label{ex:c-division example}
Let $I_d$ be the identity matrix of dimension $d\ge2$ and let $a_1,\ldots,a_{k+1}$ be nonzero scalars.
Then the row separation problem associated with the block set
$\{a_1I_d,\ldots,a_{k+1}I_d\}$ is a $(k+1)$-division problem.
In particular, it is also a complete division problem in the sense of
\Cref{def:complete division problem}.
\end{example}

\begin{example}\label{eg:4-cycle}
Let $\alpha=\sqrt{2/5}$ and $\beta=\sqrt{1/5}$, and for an angle $\theta$ let
\[
R(\theta)=\begin{bmatrix}\cos\theta & \sin\theta\\ -\sin\theta & \cos\theta\end{bmatrix}.
\]
Denote
\[
\begin{aligned}
D_1 &= \begin{bmatrix} I & 0 & 0 \end{bmatrix}, &\qquad
D_2 &= \begin{bmatrix} \alpha I & \alpha I & \beta I \end{bmatrix},\\[2pt]
D_3 &= \begin{bmatrix} 0 & I & 0 \end{bmatrix}, &\qquad
D_4 &= \begin{bmatrix} \alpha I & \alpha R(\theta_1) & \beta R(\theta_2) \end{bmatrix},
\end{aligned}
\]
where
\[
\cos\theta_1=-\tfrac{7}{8},\quad \sin\theta_1=\tfrac{\sqrt{15}}{8},
\qquad
\cos\theta_2=-\tfrac{1}{4},\quad \sin\theta_2=-\tfrac{\sqrt{15}}{4}.
\]
Then $$\bar{D_j}D_j^T=\alpha_jI\quad\text{for all } j,\qquad \bar{D_i}D_j^T\bar{D_j}D_i^T=\alpha^2_{(i,j)}I\quad\text{for all } i,j.$$
Using \Cref{le:after 2-division including}, one checks that the row separation problems associated with the block sets $\{D_1,D_2,D_3\}$, $\{D_1,D_2,D_4\}$, $\{D_1,D_3,D_4\}$ and $\{D_2,D_3,D_4\}$ are complete division problems, whereas the row separation problem associated with $\{D_1,D_2,D_3,D_4\}$ is not a complete division problem; it is only a $3$-division problem.
\end{example} 

Example~\ref{eg:4-cycle} shows that homogeneity can hold on every triple of blocks and still fail for the whole set: the obstruction is invisible to the pairwise conditions and appears only around the full cycle. To detect such obstructions, we extend the matrix conditions of the $1$- and $2$-division problems, which are stated in terms of $\bar D_jD_j^T$ and $\bar D_iD_j^T$, to products of cross-products along orderings of row blocks. Following~\cite{Possible}, we call these products covariance matrices; the term refers to the cyclic products below and is unrelated to statistical covariance.
\begin{definition}[Covariance Matrix]
\label{def:covariance matrix}
Consider a $2$-division problem associated with the block set
$\{D_1,\ \cdots,\ D_{k+1}\}$ with $k \geq 2$.
Let $\left(j_1,\ \cdots,\ j_c\right)$ be a non-repetitive and ordered index collection, where $j_i\in\{1,\ldots,k+1\}$.

The \emph{covariance matrix} associated with this ordering is defined by
\[
{\rm Cov}^{(j_1,\ldots,j_c)}
=
\bar{D}_{j_1}
D_{j_2}^T \bar{D}_{j_2}
\cdots
D_{j_c}^T \bar{D}_{j_c}
D_{j_1}^T
\]
\end{definition}

\begin{remark}
Since the problem is a $2$-division problem, every $\bar{D_i}D_j^T$ is a scalar multiple of a unitary matrix, and therefore so is ${\rm Cov}^{(j_1,\ldots,j_c)}$ for every ordering $\left(j_1,\ldots,j_c\right)$; in particular, it is a normal matrix. However, the covariance matrix ${\rm Cov}^{(j_1,\ldots,j_c)}$ is generally neither symmetric nor Hermitian. 

\end{remark}

We next relate covariance matrices to complete division problems. The
connection relies on the fact that, in a $2$-division problem, each nonzero cross-product $\bar D_iD_j^T$ is a multiple of a unitary matrix, so it can be reduced to a scalar matrix by unitary transformations within the blocks.

\begin{theorem}\label{thm:relation between eigen/singular}
Consider a $2$-division problem associated with the block set $\{D_1,\ldots,D_{k+1}\}$,
$k\ge 2$, and write $G_{ij}=\bar D_iD_j^T$ for its cross-products. Let $\mathcal{G}$ be the graph
with vertices $1,\ldots,k+1$ and an edge $\{i,j\}$ whenever $i\neq j$ and $G_{ij}\neq 0$.
The following statements are equivalent.
\begin{enumerate}
\item[(i)] The problem is a complete division problem.
\item[(ii)] There are unitary matrices $H_1,\ldots,H_{k+1}$ such that
$\bar H_iG_{ij}H_j^T$ is a scalar matrix (possibly zero) for all $i,j$.
\item[(iii)] For every non-repetitive ordering $(j_1,\ldots,j_c)$ with $3\le c\le k+1$, the
covariance matrix ${\rm Cov}^{(j_1,\ldots,j_c)}$ is a scalar matrix (possibly zero).
\item[(iv)] For one spanning tree $\mathcal{T}$ of $\mathcal{G}$, the covariance matrix of every
fundamental cycle of $\mathcal{T}$ is a scalar matrix. Here the fundamental cycle of an edge
$\{i,j\}$ of $\mathcal{G}$ not in $\mathcal{T}$ is the ordering $(i,\ell_0,\ell_1,\ldots,\ell_{m-1})$,
where $(j=\ell_0,\ell_1,\ldots,\ell_m=i)$ is the path from $j$ to $i$ in $\mathcal{T}$.
\end{enumerate}
\end{theorem}

\begin{proof}
We use two elementary facts.

(a) By \Cref{def:2d}, all blocks have the same number $r$ of rows; $G_{jj}=\alpha_jI$ with
$\alpha_j>0$; $G_{ji}=G_{ij}^*$; and for $i\neq j$ either $G_{ij}=0$ or
$G_{ij}=\rho_{ij}U_{ij}$ with $\rho_{ij}=\rho_{ji}>0$ and $U_{ij}$ unitary. In particular,
$\mathcal{G}$ is connected (Condition~3 of \Cref{def:2d}), so it has a spanning tree.

(b) Let $H_1,\ldots,H_{k+1}$ be unitary and $\hat D_j=H_jD_j$. Since
\Cref{def:row separation problem} allows arbitrary unitary matrices acting on each block, the
problems associated with $\{D_j\}$ and $\{\hat D_j\}$ have the same collection of row separation
subspaces. Their cross-products are $\hat G_{ij}=\bar H_iG_{ij}H_j^T$, and since
$H_j^T\bar H_j=I$, their covariance matrices are $\bar H_{j_1}{\rm Cov}^{(j_1,\ldots,j_c)}H_{j_1}^T$,
which are unitarily similar to the original ones. Hence each of (i), (iii) and (iv) holds for
$\{D_j\}$ if and only if it holds for $\{\hat D_j\}$; the graph $\mathcal{G}$ is the same for both.

\emph{(ii)$\Rightarrow$(i).} By (b) we may assume $H_j=I$, that is, $G_{ij}=\sigma_{ij}I$ for all
$i,j$. Let $x_j^{(p)}=D_j^Te_p$ be the $p$th row of $D_j$, written as a column vector. Then
\[
\big\langle x_j^{(p)},x_i^{(q)}\big\rangle=e_q^T\bar D_iD_j^Te_p=\sigma_{ij}\,\delta_{pq}.
\]
Hence the subspaces $\mathcal{S}_p={\rm span}\{x_1^{(p)},\ldots,x_{k+1}^{(p)}\}$, $p=1,\ldots,r$,
are mutually orthogonal. Moreover, since $\sigma_{jj}=\alpha_j>0$, the vectors
$x_j^{(1)},\ldots,x_j^{(r)}$ are nonzero and mutually orthogonal, so
$\mathcal{R}(D_j^T)=\bigoplus_{p}{\rm span}\{x_j^{(p)}\}$; as $x_j^{(q)}\perp\mathcal{S}_p$ for
$q\neq p$, this gives $\mathcal{S}_p\cap\mathcal{R}(D_j^T)={\rm span}\{x_j^{(p)}\}$ for every $j$
and $p$.

Let $W=[W_1,\ldots,W_r,W_0]$ be unitary, where the columns of $W_p$ form an orthonormal basis of
$\mathcal{S}_p$ and those of $W_0$ an orthonormal basis of the orthogonal complement of
$\mathcal{S}_1\oplus\cdots\oplus\mathcal{S}_r$, and put $R=\bar W$. For a row $x^T$ of a block with
$x\in\mathcal{S}_p$ we have $x^TR=(W^*x)^T$, whose only nonzero entries are in the columns of the
$p$th group, while the columns of the $W_0$ group are zero for every row. Hence, with the
permutation $E$ that groups the rows of all blocks according to their index $p$ and with $H_j=I$,
the matrix in \Cref{def:row separation problem} is block diagonal, the zero columns of the $W_0$
group being appended to any one of the column groups, and $\bar R$ applied to the transposed rows of the $p$th diagonal block returns
$\mathcal{S}_p$. Thus every $\mathcal{S}_p$ is a row separation subspace
(\Cref{def:row separation subspace}).

Now let $j$ be arbitrary and $0\neq v\in\mathcal{R}(D_j^T)$. Since $G_{jj}=\alpha_jI$, we can write
$v=D_j^Ty$ with $y\neq 0$. Let $H_0$ be a unitary matrix with $H_0^Te_1=y/\|y\|_2$; then the first
row of $H_0D_j$, written as a column vector, is $D_j^TH_0^Te_1=v/\|y\|_2$. Replacing every $D_i$ by
$H_0D_i$ keeps all cross-products scalar, because $\bar H_0(\sigma_{il}I)H_0^T=\sigma_{il}I$, and by
(b) it does not change the collection of row separation subspaces. Applying the preceding
paragraphs to $\{H_0D_1,\ldots,H_0D_{k+1}\}$ gives a row separation subspace $\mathcal{S}_1$ with
$v\in\mathcal{S}_1$ and $\dim\big(\mathcal{S}_1\cap\mathcal{R}(D_i^T)\big)=1$ for
$i=1,\ldots,k+1$. This is Condition~2 of \Cref{def:c-division} with $c=k+1$, so the problem is a
complete division problem.

\emph{(i)$\Rightarrow$(iii).} Let $(j_1,\ldots,j_c)$ be a non-repetitive ordering with
$3\le c\le k+1$ and let $\{C\}=\{j_1,\ldots,j_c\}$. A complete division problem is a $c$-division
problem (see the remark following \Cref{def:c-division}). Take any $0\neq y\in\mathbb{F}^r$ and
put $v=D_{j_1}^Ty\neq 0$. By \Cref{def:c-division} there is a row separation subspace
$\mathcal{S}$ of the $\{C\}$-subset subproblem such that $v\in\mathcal{S}$ and
$\dim\big(\mathcal{S}\cap\mathcal{R}(D_{j_1}^T)\big)=1$, hence
$\mathcal{S}\cap\mathcal{R}(D_{j_1}^T)={\rm span}\{v\}$. By \Cref{le:D_j and S cap}, applied to the
$\{C\}$-subset subproblem, $D_j^T\bar D_j\mathcal{S}\subseteq\mathcal{S}$ for every
$j\in\{C\}$. Therefore
\[
w:=D_{j_1}^T\bar D_{j_1}\,D_{j_2}^T\bar D_{j_2}\cdots D_{j_c}^T\bar D_{j_c}\,v
\in\mathcal{S}\cap\mathcal{R}(D_{j_1}^T)={\rm span}\{v\}.
\]
Since $w=D_{j_1}^T\,{\rm Cov}^{(j_1,\ldots,j_c)}y$ and $D_{j_1}^T$ is injective, $y$ is an
eigenvector of ${\rm Cov}^{(j_1,\ldots,j_c)}$. As $y\neq 0$ was arbitrary, every nonzero vector
is an eigenvector. If $y$ and $z$ are linearly independent, with eigenvalues $\mu_y$ and $\mu_z$,
then $y+z$ is also an eigenvector, say with eigenvalue $\mu$, and
$\mu_yy+\mu_zz=\mu(y+z)$ gives $\mu_y=\mu=\mu_z$. Hence all eigenvalues coincide and
${\rm Cov}^{(j_1,\ldots,j_c)}$ is a scalar matrix.

\emph{(iii)$\Rightarrow$(iv).} The covariance matrix of a fundamental cycle is
${\rm Cov}^{(j_1,\ldots,j_c)}$ for a non-repetitive ordering with $3\le c\le k+1$.

\emph{(iv)$\Rightarrow$(ii).} Root $\mathcal{T}$ at the vertex $1$. Put $H_1=I$ and, proceeding
from the root, $H_b=H_a\bar U_{ab}$ whenever $a$ is the parent of $b$ in $\mathcal{T}$. Then
\[
\bar H_aG_{ab}H_b^T=\rho_{ab}\bar H_aU_{ab}U_{ab}^*H_a^T=\rho_{ab}I,
\qquad
\bar H_bG_{ba}H_a^T=\big(\bar H_aG_{ab}H_b^T\big)^*=\rho_{ab}I .
\]
Let $\hat D_j=H_jD_j$ and $\hat G_{ij}=\bar H_iG_{ij}H_j^T$. Then $\hat G_{jj}=\alpha_jI$,
$\hat G_{ij}=0$ if $\{i,j\}$ is not an edge of $\mathcal{G}$, and $\hat G_{ab}=\rho_{ab}I$ for
every edge $\{a,b\}$ of $\mathcal{T}$, in either orientation. By (b), (iv) holds for
$\{\hat D_j\}$. Let $\{i,j\}$ be an edge of $\mathcal{G}$ not in $\mathcal{T}$, with tree path
$(j=\ell_0,\ldots,\ell_m=i)$, $m\ge 2$. The covariance matrix of its fundamental cycle is
\[
\widehat{\rm Cov}^{(i,\ell_0,\ldots,\ell_{m-1})}
=\hat G_{ij}\,\hat G_{\ell_0\ell_1}\cdots\hat G_{\ell_{m-1}\ell_m}
=\Big(\prod_{t=0}^{m-1}\rho_{\ell_t\ell_{t+1}}\Big)\,\hat G_{ij}.
\]
It is a scalar matrix and the product of the $\rho$'s is positive, so $\hat G_{ij}$ is a scalar
matrix. Hence every $\hat G_{ij}$ is a scalar matrix, which is (ii).
\end{proof}

\Cref{thm:relation between eigen/singular} characterizes complete homogeneity. We now ask where it first fails: for a $c$-division problem that is not a $(c+1)$-division problem, we show that the inhomogeneity is concentrated on a chordless cycle.

By \Cref{def:c-division}, there is an index subset $\{C\}$ with
$|\{C\}|=c+1$ that violates Condition~2. The $\{C\}$-subset subproblem is connected: otherwise each of its connected components has at most $c$ blocks and satisfies Condition~2, and the sum of the corresponding row separation subspaces, which are mutually orthogonal, would satisfy it for $\{C\}$. Hence this subproblem is a $2$-division problem but not a complete division problem. Since all its proper subsets have at most $c$ blocks, their covariance matrices are scalar matrices, and by \Cref{thm:relation between eigen/singular} there is an ordering
$(j_1,\dots,j_{c+1})$ of $\{C\}$ such that ${\rm Cov}^{(j_1,\dots,j_{c+1})}$ is not a scalar matrix. Restricting to this subproblem and relabeling, we may
assume $k=c$ and $(j_1,\dots,j_{c+1})=(1,\dots,k+1)$.

\begin{theorem} \label{thm:single cycle}
Consider a $k$-division problem with blocks $D_1,\ldots,D_{k+1}$ such that ${\rm Cov}^{(1,\ldots,k+1)}=\bar{D_1}D_2^T\cdots\bar{D}_kD^T_{k+1}\bar{D}_{k+1}D_1^T$ is nonzero and not a scalar matrix. Then $\bar{D}_iD_j^T=0$ whenever $|j-i|>1$ and $\{i,j\} \neq \{1, k+1\}$.
\end{theorem}
\begin{proof}
Suppose that there are indices $i<j$ with $j-i>1$, $\{i,j\} \neq\{ 1, k+1\}$ and $\bar{D}_iD_j^T\neq0$. Then
$${\rm Cov}^{\left(i, \ i+1,\ \cdots,\ j-1,\ j\right)}=\bar{D_i}D_{i+1}^T\cdots\bar{D}_{j-1}D_{j}^T\bar{D_j}D_i^T\neq0,$$
so the $\{i,\ldots,j\}$-subset subproblem is a $2$-division problem by \Cref{def:2d}. Moreover, $k>j-i>1$, so this subset subproblem has at most $k$ blocks and is therefore a complete division problem. By \Cref{thm:relation between eigen/singular}, ${\rm Cov}^{\left(i, \ i+1,\ \cdots,\ j-1,\ j\right)}=\bar{D_i}D_{i+1}^T\cdots\bar{D}_{j-1}D_{j}^T\bar{D_j}D_i^T=\alpha I$. Therefore, $\bar{D_i}D_{i+1}^T\cdots\bar{D}_{j-1}D_{j}^T=\frac{\alpha}{||\bar{D_i}D_j^T||^2_2}\bar{D_i}D_j^T\neq 0$.

$$\begin{aligned}
{\rm Cov}^{\left(1,\ \cdots,\ k+1\right)}&=\bar{D_1}D_2^T\cdots \bar{D}_{i-1}D_{i}^T\cdots\bar{D}_{j}D_{j+1}^T\cdots \bar{D}_{k+1}D_1^T\\
&=\bar{D_1}D_2^T\cdots \bar{D}_{i-1}D_{i}^T\frac{\alpha}{||\bar{D_i}D_j^T||^2_2}\bar{D_i}D_j^T\bar{D}_{j}D_{j+1}^T\cdots \bar{D}_{k+1}D_1^T\\
&=\frac{\alpha}{||\bar{D_i}D_j^T||^2_2}{\rm Cov}^{\left(1,\ \cdots,\ i,\ j,\ \cdots,\ k+1\right)}.
\end{aligned}$$
Since ${\rm Cov}^{\left(1,\ \cdots,\ k+1\right)}\neq 0$, also ${\rm Cov}^{\left(1,\ \cdots,\ i,\ j,\ \cdots,\ k+1\right)}\neq 0$, so by \Cref{def:2d} the $\{1,\ldots,i,j,\ldots,k+1\}$-subset subproblem is a $2$-division problem with $k+2-j+i\leq k$ blocks, hence a complete division problem.
Therefore ${\rm Cov}^{\left(1,\ \cdots,\ i,\ j,\ \cdots,\ k+1\right)}$ is a scalar matrix by \Cref{thm:relation between eigen/singular}, and so is ${\rm Cov}^{\left(1,\ \cdots,\ k+1\right)}$, a contradiction.
\end{proof}
\begin{remark}
\Cref{thm:single cycle} shows that the inhomogeneity is carried by a single chordless cycle. This simple structure is what allows the direct splitting and equivalent replacement in
\Cref{sec:division problem construction in complex number}.

\end{remark}

\begin{corollary}    
\label{thm:chordless}
Consider a $2$-division problem with row blocks $D_1,\dots,D_{k+1}$, $k\ge3$.
If there is an ordering $(j_1,\dots,j_{k+1})$ of $\{1,\dots,k+1\}$
such that
\[
  \bar D_{j_i}D_{j_s}^{T} = 0
  \quad\text{whenever } |i-s|>1 \text{ and } \{i,s\}\neq\{1,k+1\},
\]
then the problem is a $k$-division problem. 
\end{corollary}
\begin{proof}
By the remark following \Cref{def:c-division}, it suffices to verify Condition~2 of \Cref{def:c-division} for every index subset $\{C\}$ with $k$ elements. 
By the hypothesis and connectivity, the graph of nonzero cross-products is either the chordless cycle on all $k+1$ blocks or a spanning path obtained by deleting one cycle edge. In either case, its restriction to $\{C\}$ is a disjoint union of paths, and blocks in different paths have zero cross-products, hence mutually orthogonal row spaces.
Let $E_1,\ldots,E_q$ be the blocks of one of these paths, in path order. Normalizing successively along the path, with $H_1=I$ and $H_{i+1}=H_i\bar U_{i,i+1}$ where $\bar E_iE_{i+1}^T=\rho_iU_{i,i+1}$ with $\rho_i>0$ and $U_{i,i+1}$ unitary, makes every nonzero cross-product of the path a scalar matrix, since only consecutive blocks interact. Replacing all $H_iE_i$ by $H_0H_iE_i$ for a common unitary $H_0$ preserves this property, since $\bar H_0(\sigma I)H_0^T=\sigma I$. If a prescribed nonzero vector is $w=(H_iE_i)^Ty$ for some block $i$ of the path, choosing $H_0$ with $H_0^Te_1=y/\|y\|_2$ makes $w/\|y\|_2$ the first row of $H_0H_iE_i$, written as a column vector. The span of the first rows of the blocks $H_0H_iE_i$ is then invariant under all Gram operators of the path, because rows with different indices are orthogonal; by \Cref{le:D_j and S cap}(2), it is a row separation subspace of the path, and it meets the row space of every block of the path in a one-dimensional subspace.

Now let $v\in\mathcal{R}(D_{j}^T)$, $j\in\{C\}$, be nonzero. Choose such a subspace containing $v$ for the path that contains $j$, and such a subspace for an arbitrary nonzero vector in each of the other paths. These subspaces are mutually orthogonal, so their sum is a row separation subspace of the $\{C\}$-subset subproblem by \Cref{cor:row separation subspace sum}; it contains $v$ and meets the row space of every block of $\{C\}$ in a one-dimensional subspace. Hence Condition~2 holds for $\{C\}$, and the problem is a $k$-division problem.
\end{proof}
\begin{remark}

\Cref{thm:chordless} 
covers the chordless-cycle sparsity pattern; since the interaction graph is connected, it is either the full chordless cycle or the spanning path obtained by deleting one cycle edge. In particular, a chordless cycle of length k+1 automatically satisfies the division conditions of all levels below its own, i.e., it is a $k$-division problem. This allows us to search for chordless cycles locally, without regard to the homogeneity level $c$ of the whole problem.

\end{remark}
\section{Solutions of the Row Separation Problem via Unitary Transformations}
\label{sec:division problem construction in complex number}

\Cref{sec:division problem theory} identified the obstruction to complete division: a chordless cycle whose covariance matrix is not a scalar matrix (\Cref{thm:relation between eigen/singular} and \Cref{thm:single cycle}). In this section we turn this obstruction into a splitting procedure: we split the blocks of such a cycle along the eigenspaces of its covariance matrix and repeat until every remaining piece either has a single indivisible subspace or is a complete division problem. In the latter case, the normal form of \Cref{thm:relation between eigen/singular} makes its indivisible subspaces explicit (Remark~\ref{rem:output}). In this way the strategy obtains all indivisible subspaces.

The strategy is formulated over $\mathbb{C}$, in line with the target
decomposition fixed in \Cref{def:red}. For $c\ge3$, it splits row blocks along individual eigenspaces of the covariance matrices
${\rm Cov}^{(j_1,\dots,j_c)}$, which are normal but in general not Hermitian and may have complex conjugate eigenvalues even for real data (Example~\ref{eg:4-cycle}). Working over $\mathbb{C}$ allows each such eigenspace to be treated separately.

We thus seek the finest decomposition over $\mathbb{C}$; this is a choice of target. By \Cref{def:dense problem}, the complex decomposition is at least as fine as the real one and may be strictly finer, and real data are covered as a special case. The price is that complex arithmetic is required even for real
data. Over $\mathbb{R}$, the finest decomposition of a core problem is given by the simultaneous singular value decomposition of Maehara and Murota~\cite[Thm.~7 and Cor.~8]{MaeharaMurota2011LAA}, and their algorithms use real arithmetic only~\cite{MaeharaMurota2010,MaeharaMurota2011LAA}, thereby avoiding conjugate eigenpairs. A real-arithmetic version of the row separation strategy itself is left for future work.

The strategy developed in this section connects the unitary decomposition of TLS core problems with the graph structure of the interactions among row blocks: the components are determined by the spectral structure of the covariance operators along cycles of this graph. The exact strategy and the underlying theory are intended as a reference against which efficient iterative algorithms and algorithms for the real field can be compared and
validated. A small floating-point example is presented to illustrate how the strategy runs. More general numerical experiments, such as timings, scalability and the effect of noise, as well as computational optimization, are beyond the scope of this paper.

This section is organized as follows. In \Cref{sec:alternative construction}, we replace $\{C\}$-subset subproblems that fail to be ideally homogeneous by equivalent problems that preserve the row separation subspaces and restore a tractable structure. In \Cref{sec:final strategy}, this construction is
embedded into a recursive strategy that yields all indivisible subspaces of a row separation problem over $\mathbb{C}$. \Cref{sec:ssvd} relates the decomposition to the simultaneous singular value decomposition, which yields its uniqueness up to unitary equivalence and permutation of the components, and describes the structure of the output. \Cref{sec:example} illustrates the strategy on a small floating-point example with real data.

\subsection{Alternative Construction}\label{sec:alternative construction}
Starting from a $k$-division problem $\mathsf{P}$ with $D_j\in\mathbb{C}^{r_j\times \hat d}$ ($k\geq2$) that is not a $(k+1)$-division problem, we show how to replace a non-ideally homogeneous
$\{C\}$-subset subproblem by an equivalent row separation problem with improved structural properties, without altering the
collection of row separation subspaces.
All arguments in this subsection are carried out over the complex field~$\mathbb{C}$.

Since $\mathsf{P}$ is not a $(k+1)$-division problem, the argument preceding \Cref{thm:single cycle} gives a non-repetitive ordering $(j_1,\dots,j_{k+1})$ whose covariance matrix is nonzero and not a scalar matrix, and \Cref{thm:single cycle} applies to it. Without loss of generality, let
$(j_1,\dots,j_{k+1})=(1,\dots,k+1)$. Since left multiplication of each block by a unitary matrix does not change the row separation subspaces, we may further assume, applying the construction in the proof of \Cref{thm:relation between eigen/singular} along the path
$1,2,\dots,k+1$, that
\[
  \bar D_iD_{i+1}^T=\alpha_{(i,i+1)}I,\qquad \alpha_{(i,i+1)}>0,
  \qquad i=1,\dots,k,
\]
where $\bar{D}_{k+1}D^T_{1}$ is a nonzero scalar multiple of a unitary matrix and
${\rm Cov}^{(1,\dots,k+1)}=\alpha_{(1,2)}\cdots\alpha_{(k,k+1)}\bar{D}_{k+1}D_1^T$.

Suppose that ${\rm Cov}^{(1,\dots,k+1)}$ has $l$ distinct eigenvalues $\lambda_1,\dots,\lambda_l$ with algebraic multiplicities $f_1,\dots,f_l$($l\geq2$). Since the matrix is normal (see the remark following \Cref{def:covariance matrix}), its eigenspaces are mutually orthogonal and $f_s$ equals the dimension of the $s$th eigenspace.
Let the columns of $Y_s\in\mathbb{C}^{r\times f_s}$ form an orthonormal basis of this eigenspace; then
$Y=\left[Y_1,\ldots,Y_l\right]\in\mathbb{C}^{r\times r}$ is unitary.

Then define:
\[
D^s_{1}=Y_s^TD_{1},\qquad
D^s_{i}=\frac{1}{\bar{\alpha}_{(i-1,i)}}D^s_{i-1}D_{i}^*D_{i},\quad 2\le i\le k+1 ,
\]
where $\alpha_{(i,i+1)}I=\bar{D_i}D_{i+1}^T$.
Moreover,
\begin{equation}\label{eq:Ds simple}
D^s_{i}=
Y_s^TD_i,\ \quad i=2,\ \cdots,\ k+1,\ s=1,\ \cdots,\ l.
\end{equation}

Since $\mathsf{P}$ is a $2$-division problem, $\bar D_iD_i^T=\alpha_iI$ with $\alpha_i>0$ (\Cref{def:1d}), and by the definition of $D_1^s$ and \eqref{eq:Ds simple}
\[
\bar D_i^s\big(D_i^s\big)^T=Y_s^*\,\bar D_iD_i^T\,Y_s=\alpha_iI_{f_s};
\]
hence each $D_i^s$ has full row rank $f_s$, and the row blocks below satisfy the standing assumption of \Cref{sec:row separation problem}.
Using these matrices, we define the following row separation problems:
\begin{itemize}
\item $\mathsf{P}_s$: the row separation problem constructed from the block set
$\{D^{s}_{1},\ \cdots,\ D^{s}_{k+1}\}$.
\item $\mathsf{P}^{\rm Com}$: the row separation problem associated with the union of the block sets $\{D_1^s,\ldots,D_{k+1}^s\}$, $s=1,\ldots,l$; this is the row block composition of $\mathsf{P}_{1},\ldots,\mathsf{P}_{l}$ (\Cref{def:row block composition}).
\end{itemize}
\begin{remark}\label{rem:normalization}
The preliminary normalization and reordering of the $D_j$ at the beginning of this subsection does not change $\mathsf{P}_s$. Replacing $D_j$ by $H_jD_j$ with $H_j$ unitary and $H_1=I$ leaves $D_1$ and every product $D_j^*D_j$ unchanged, and it also leaves ${\rm Cov}^{(1,\dots,k+1)}$ unchanged since
it is transformed into $\bar H_1{\rm Cov}^{(1,\dots,k+1)}H_1^T$. Moreover, the scalars $\alpha_{(i-1,i)}=\|\bar D_{i-1}D_i^T\|_2>0$ in the recursion are unchanged, since $\bar H_{i-1}\bar D_{i-1}D_i^TH_i^T$ has the same singular values as $\bar D_{i-1}D_i^T$. Hence the matrices $Y_s$ and the recursion defining $D_i^s$ are the same before and after
the normalization, and so are the blocks $D_i^s$ and the problem $\mathsf{P}_s$; the normalization only makes \eqref{eq:Ds simple} available.

\end{remark}

\begin{corollary}
\label{thm:P+ complete division}
Let $\mathsf{P}_s$ be defined above and let $f_s$ denote the algebraic multiplicity of the eigenvalue $\lambda_s$ of ${\rm Cov}^{(1,\ \cdots,\ k+1)}$. Then:
\begin{enumerate}
\item $\forall\ s,\ s^{'}$ and $\forall\ i,\ i^{'}$, if $s\neq s^{'}$,
\[ \bar{D}^{s}_{i}\left(D^{s^{'}}_{i^{'}}\right)^T=0.\]
\item If $f_s=1$, the problem $\mathsf{P}_s$ has exactly one indivisible subspace $\mathcal{I}^{\mathsf{P}_s}$.
\item The problem $\mathsf{P}_s$ is a complete division problem if and only if $f_s \ge 2$.
\end{enumerate}
\end{corollary}
\begin{proof}
Since $\mathsf{P}$ is a $k$-division problem with $k\ge2$, it is a $2$-division problem; thus $\bar D_iD_i^T=\alpha_iI$ with $\alpha_i>0$ for every $i$ (\Cref{def:1d}). By \Cref{thm:single cycle}, $\bar D_iD_{i'}^T=0$ whenever $|i-i'|>1$ and $\{i,i'\}\neq\{1,k+1\}$. Moreover, ${\rm Cov}^{(1,\dots,k+1)}=\alpha_{(1,2)}\cdots\alpha_{(k,k+1)}\,\bar D_{k+1}D_1^T$, so that
\[
\bar D_{k+1}D_1^TY_s=\frac{\lambda_s}{\alpha_{(1,2)}\cdots\alpha_{(k,k+1)}}\,Y_s,\qquad s=1,\dots,l,
\]
and $\lambda_s\neq0$, because ${\rm Cov}^{(1,\dots,k+1)}$ is a product of nonzero scalar multiples of unitary matrices. By the definition of $D_1^s$ and \eqref{eq:Ds simple}, for all $s,s',i,i'$,
\[
\bar D_i^s\big(D_{i'}^{s'}\big)^T=Y_s^*\,\bar D_iD_{i'}^T\,Y_{s'} .
\]
Since $Y_s^*Y_s=I_{f_s}$ and $Y_s^*Y_{s'}=0_{f_s\times f_{s'}}$ for $s\neq s'$, this gives
\begin{equation}\label{eq:Ps cross-products}
\bar D_i^s\big(D_{i'}^{s'}\big)^T=
\begin{cases}
\alpha_i\,I_{f_s}, & s=s',\ i'=i,\\[2pt]
\alpha_{(i,i+1)}\,I_{f_s}, & s=s',\ i'=i+1\le k+1,\\[2pt]
\dfrac{\lambda_{s}}{\alpha_{(1,2)}\cdots\alpha_{(k,k+1)}}\,I_{f_s}, & s=s',\ (i,i')=(k+1,1),\\[6pt]
0_{f_s\times f_s}, & s=s',\ |i-i'|>1,\ \{i,i'\}\neq\{1,k+1\},\\[2pt]
0_{f_s\times f_{s'}}, & s\neq s',
\end{cases}
\end{equation}
and the remaining cases follow by conjugate transposition, since $\bar D_{i'}^{s'}(D_i^s)^T=\big(\bar D_i^s(D_{i'}^{s'})^T\big)^*$.

Item (1) is the last case of \eqref{eq:Ps cross-products}.

For (2), let $f_s=1$. Then each $D_i^s$ is a single row; write $x_i=(D_i^s)^T$, so that the row-space Gram operators of $\mathsf{P}_s$ are $x_ix_i^*$. By \eqref{eq:Ps cross-products}, $x_i\neq0$ and $\langle x_i,x_{i+1}\rangle\neq0$ for $i=1,\dots,k$. Let $\mathcal{S}$ be a row separation subspace of $\mathsf{P}_s$. By \Cref{le:D_j and S cap}, $\mathcal{S}\neq\{0\}$, $\mathcal{S}\subseteq{\rm span}\{x_1,\dots,x_{k+1}\}$, and $x_ix_i^*\mathcal{S}\subseteq\mathcal{S}$ for every $i$; the last condition means that either $x_i\in\mathcal{S}$ or $x_i\perp\mathcal{S}$. Since $\mathcal{S}$ is a nonzero subspace of ${\rm span}\{x_1,\dots,x_{k+1}\}$, some $x_i$ is not orthogonal to $\mathcal{S}$, hence $x_i\in\mathcal{S}$. If $x_i\in\mathcal{S}$, then $x_{i\pm1}$ is not orthogonal to $\mathcal{S}$, hence $x_{i\pm1}\in\mathcal{S}$. Proceeding along the path $1,2,\dots,k+1$, all $x_i$ lie in $\mathcal{S}$, so $\mathcal{S}={\rm span}\{x_1,\dots,x_{k+1}\}$. Thus $\mathsf{P}_s$ has exactly one row separation subspace, which is therefore its unique indivisible subspace $\mathcal{I}^{\mathsf{P}_s}$.

For (3), if $f_s=1$, the blocks $D_i^s$ have rank one, so $\mathsf{P}_s$ violates Condition~2 of \Cref{def:1d}; it is not a $1$-division problem and hence not a complete division problem. If $f_s\ge2$, then $\bar D_i^s(D_i^s)^T=\alpha_iI_{f_s}$ with $\alpha_i>0$ and ${\rm rank}(D_i^s)=f_s\ge2$, so $\mathsf{P}_s$ is a $1$-division problem; every cross-product $\bar D_i^s(D_{i'}^s)^T$ is a scalar matrix, in particular a scalar multiple of a unitary matrix; and the nonzero cross-products between consecutive blocks connect all blocks. Hence $\mathsf{P}_s$ is a $2$-division problem whose cross-products are all scalar matrices, and it is a complete division problem by \Cref{thm:relation between eigen/singular}, (ii)$\Rightarrow$(i), with $H_i=I$.
\end{proof}

\begin{remark}
When $f_s=1$, the blocks $D_i^s$ have rank one and therefore fall outside the division hierarchy of \Cref{def:1d}. This case is handled separately by Step~1.3 of \Cref{alg:strategy} and yields a single unambiguous indivisible subspace. When $f_s\ge2$, the problem $\mathsf{P}_s$ is a complete division problem and is handled by the output step (Remark~\ref{rem:output}).
\end{remark}

The following theorem establishes the equivalence, in terms of row separation subspaces, between the original problem $\mathsf{P}$ and the problem $\mathsf{P}^{\rm Com}$ constructed above. 

\begin{theorem}\label{thm:Pcom P equivalence}
Let $\mathsf{P}_s$, $s=1,\ldots,l$, and $\mathsf{P}^{\rm Com}$ be defined above. Then $\mathsf{P}$ and $\mathsf{P}^{\rm Com}$ are equivalent with respect to row separation subspaces.
\end{theorem}

\begin{proof}
By Remark~\ref{rem:normalization}, we may work with the normalized blocks, for which \eqref{eq:Ds simple} holds; by definition, the blocks of $\mathsf{P}^{\rm Com}$ are the $D_i^s$, $s=1,\dots,l$, $i=1,\dots,k+1$. Let $\Pi_s=Y_sY_s^*$ be the orthogonal projector onto the $s$th eigenspace of ${\rm Cov}={\rm Cov}^{(1,\dots,k+1)}$, and denote the row-space Gram operators of $\mathsf{P}$ and $\mathsf{P}^{\rm Com}$ by
\[
K_i=D_i^T\bar D_i,\qquad K_i^s=\big(D_i^s\big)^T\overline{D_i^s}=D_i^T\,\Pi_s\,\bar D_i,
\]
where the last identity follows from the definition of $D_1^s$ and \eqref{eq:Ds simple}. Since $\sum_s\Pi_s=I$, we have $\sum_sK_i^s=K_i$, and $\sum_s\mathcal{R}\big((D_i^s)^T\big)=\mathcal{R}(D_i^TY)=\mathcal{R}(D_i^T)$ because $Y$ is unitary; hence the row spaces of all blocks of $\mathsf{P}$ and of $\mathsf{P}^{\rm Com}$ have the same sum. By \Cref{le:D_j and S cap}, it therefore suffices to show that a subspace is invariant under all $K_i$ if and only if it is invariant under all $K_i^s$.

If a subspace is invariant under all $K_i^s$, it is invariant under their sums $K_i$. For the converse, let $\mathcal{A}$ denote the complex algebra generated by $K_1,\dots,K_{k+1}$, that is, the set of all complex linear combinations of finite products of these matrices; it suffices to show that every $K_i^s$ belongs to $\mathcal{A}$.

First let $i=1$. Since $\bar D_1D_1^T=\alpha_1I$, the map $\Phi(X)=D_1^TX\bar D_1$ satisfies $\Phi(X)\Phi(Z)=\alpha_1\Phi(XZ)$ and $\Phi(I)=K_1$, and
\[
\Phi({\rm Cov})=D_1^T\bar D_1D_2^T\bar D_2\cdots D_{k+1}^T\bar D_{k+1}D_1^T\bar D_1=K_1K_2\cdots K_{k+1}K_1\in\mathcal{A}.
\]
Hence $\Phi({\rm Cov}^n)=\alpha_1^{1-n}\Phi({\rm Cov})^n\in\mathcal{A}$ for $n\ge1$, and $\Phi(p({\rm Cov}))\in\mathcal{A}$ for every polynomial $p$. Since ${\rm Cov}$ is normal, its spectral projectors are polynomials in ${\rm Cov}$,
\[
\Pi_s=\prod_{t\neq s}\frac{{\rm Cov}-\lambda_tI}{\lambda_s-\lambda_t},
\]
so $K_1^s=\Phi(\Pi_s)\in\mathcal{A}$.

Now let $1\le i\le k$ and assume $K_i^s\in\mathcal{A}$. Since $\bar D_{i+1}D_i^T=\big(\bar D_iD_{i+1}^T\big)^*=\alpha_{(i,i+1)}I$,
\[
K_{i+1}K_i^sK_{i+1}=D_{i+1}^T\big(\bar D_{i+1}D_i^T\big)\Pi_s\big(\bar D_iD_{i+1}^T\big)\bar D_{i+1}=\alpha_{(i,i+1)}^2\,K_{i+1}^s ,
\]
so $K_{i+1}^s\in\mathcal{A}$. By induction, all $K_i^s$ belong to $\mathcal{A}$, which completes the proof.
\end{proof}

This result shows that a $\{C\}$-subset subproblem with non-ideal homogeneity can be replaced by an equivalent
row separation problem whose subproblems $\mathsf{P}_s$ either have a single indivisible subspace ($f_s=1$) or are complete division problems ($f_s\ge2$), by \Cref{thm:P+ complete division}, without altering the underlying row separation subspaces.
However, performing such replacements locally may destroy the global division structure of the problem.
To address this issue, we introduce in the next subsection a recursive strategy that restores global consistency
while systematically extracting all indivisible subspaces.

\subsection{Final Strategy}\label{sec:final strategy}
The construction of a row separation problem associated with a given core problem is described in \cite{Possible} and recalled in \Cref{sec:pre}. Moreover, once the indivisible subspaces of a row separation problem are known, all the corresponding component subproblems can be reconstructed explicitly.
We now describe the final strategy for computing all indivisible subspaces of a row separation problem.
The procedure is organized into successive division cycles and operates entirely over the complex field~$\mathbb{C}$. At each step, the procedure replaces a non-ideally homogeneous subproblem by an equivalent one with improved division structure, without altering the underlying row separation subspaces.

\begin{algorithm}[ht]
\caption{Strategy for computing all indivisible subspaces}\label{alg:strategy}
\begin{algorithmic}[1]
\Require Row blocks $\{D_1,\dots,D_{k+1}\}$ over $\mathbb{C}$ with full row rank
\Ensure All indivisible subspaces
\State Split each $D_j$ by its distinct singular values \Comment{Step 1.1}
\If{some block consists of a single row $d^T$} \label{alg:s12} \Comment{Step 1.2}
  \State \textbf{output} $\mathcal{M}(d)$ and remove the corresponding subproblem \Comment{Step 1.3}
  \State \textbf{go to} line~\ref{alg:s12}
\EndIf
\If{there exist $i,j$ such that $\bar D_iD_j^T$ is not a scalar (including zero) multiple of unitary matrix} \Comment{Step 2.1}
  \State Split $D_i,D_j$ by the corresponding left and right singular subspaces \Comment{Step 2.2}
  \State \textbf{go to} line~\ref{alg:s12}
\EndIf
\State Divide the row blocks into mutually orthogonal groups and apply lines 11–19 independently to each group \Comment{Step 2.3}
\State $c\gets 3$
\While{$c\le$ the current number of row blocks} \label{alg:c}
  \If{some ${\rm Cov}^{(j_1,\dots,j_c)}$ is not a scalar matrix} \Comment{Step c.1}
    \State Split $D_{j_1},\dots,D_{j_c}$ as in \Cref{sec:alternative construction} \Comment{Step c.2}
    \State \textbf{go to} line~\ref{alg:s12}
  \EndIf
  \State $c\gets c+1$
\EndWhile
\State \textbf{output} complete division problem \label{alg:out} \Comment{output step}
\end{algorithmic}

\end{algorithm}

\begin{remark}\label{rem:fp scaling}\textbf{Splitting Method:}
The extraction in Step~1.3 can be carried out by the construction of \Cref{thm:extend of existence of minimized}. The splittings in Steps~1.1, 2.2 and~c.2 amount to multiplying a block from the left by a unitary matrix and grouping its rows; see \Cref{cor:1d equivalent}, \Cref{cor:2d equivalent} and \Cref{thm:Pcom P equivalence}. In floating-point arithmetic, the covariance matrices $\mathrm{Cov}^{(j_1,\dots,j_c)}$, being products of $c$ cross-products, tend to become small as $c$ grows, and the same may happen to some cross-products $\bar D_iD_j^{T}$. Their eigenvectors and singular vectors are unaffected by scaling, but the decisions whether such a matrix is zero or a scalar matrix are made against a tolerance, and small magnitudes make these decisions less reliable. We therefore suggest multiplying the blocks $D_j$ by suitable positive constants so that the cross-products remain of moderate size. Such scaling does not change the row separation subspaces of the problem, see \Cref{def:row separation problem}, and in exact arithmetic it does not change any decision of the strategy.
\end{remark}

\begin{remark}\label{rem:output}\textbf{Output:}
In \Cref{alg:strategy}, Step~1.3 outputs unambiguous indivisible subspaces (see~\cite[Sec.~3.2.1]{Possible}),
whereas the output step returns complete division problems, which contain only mutually interacting ambiguous indivisible subspaces (see~\cite[Sec.~3.2.2]{Possible}). In the latter, after the normal form of \Cref{thm:relation between eigen/singular} has been applied, the $p$th rows of $D_1,\dots,D_{k+1}$ span an indivisible subspace for $p=1,\dots,r$. Indeed, a row separation subspace contained in their span is spanned by some of these rows, and connectivity forces it to contain all of them, as in the proof of \Cref{thm:P+ complete division}.
\end{remark}

\begin{remark}\textbf{Validity of Step c.2:}
In Step~c.1 the search proceeds by increasing cycle length. This guarantees that every cycle reaching Step~c.2 is a $(c-1)$-division problem, so the $\{j_1,\dots,j_c\}$-subset found in Step~c.1 necessarily satisfies the chordless-cycle condition of \Cref{thm:single cycle}. In Step~c.2 the blocks $D_{j_1},\dots,D_{j_c}$ are split and reordered as in \Cref{sec:alternative construction}, which yields a new subproblem with the same indivisible subspaces as the original $\{j_1,\dots,j_c\}$-subset subproblem. By \Cref{thm:subset exchange}, this local equivalent replacement leaves the row separation subspaces of the whole problem unchanged.
\end{remark}
\begin{remark}\textbf{Termination:}
Let $\hat m$ denote the total number of rows of the current problem and $\hat k+1$ its number of blocks. Each block has at least one row, so $\hat m-\hat k\ge1 $, and
\[
Q=(\hat m-\hat k)\,\hat m
\]
is a positive integer.  

The strategy returns to Step 1.2 only after Step~1.3, Step~2.2 or Step~c.2, and each such return strictly decreases $Q$. In Step~2.2 and Step~c.2, at least one block is split into two or more blocks, so $\hat k$ increases while $\hat m$ is unchanged, and $\hat m-\hat k$ decreases. In Step~1.3, $p\ge 1$ rows are removed and at most $p$ blocks disappear, so $\hat m$ decreases while $\hat m-\hat k$ does not increase.
Since line 12 requires $c$ not to exceed the current number of blocks $\hat{k}+1\leq m$, Step~c.1 checks finitely many values of $c$ and finitely many orderings for each. Therefore the strategy terminates.
\end{remark}

\begin{remark}\textbf{Potential High Cost:}
Step c.1 is potentially the most expensive part of the strategy. As stated, it searches candidate cycles of increasing length, and the number of candidates can grow exponentially with the number of blocks; in particular, a graph may contain exponentially many chordless cycles.
\end{remark}

\subsection{Relation to simultaneous SVD and structure of the output}\label{sec:ssvd}
\Cref{alg:strategy} works on the side of the row blocks, through the Gram operators $D_j^T\bar D_j$. In this subsection we view the resulting decomposition from the left side of the core problem, where it is described by $B_1B_1^*$ and $A_{11}A_{11}^*$. This establishes a correspondence with a finest simultaneous singular value decomposition in the sense of Maehara and Murota~\cite{MaeharaMurota2011LAA}, whose theory gives its existence and its uniqueness up to unitary equivalence and permutation of the components, and it shows how multiplicities appear in the output of \Cref{alg:strategy}.

Throughout this subsection, $[B_1\mid A_{11}]$ is a core problem with $m$ rows, so that $B_1$ and $A_{11}$ have full column rank, and $\mathcal{A}\subseteq\mathbb{C}^{m\times m}$ denotes the unital complex $*$-algebra generated by $B_1B_1^*$ and $A_{11}A_{11}^*$. Since the generators are Hermitian, a subspace is $\mathcal{A}$-invariant if and only if it is invariant under $B_1B_1^*$ and $A_{11}A_{11}^*$, and the orthogonal complement of an $\mathcal{A}$-invariant subspace is $\mathcal{A}$-invariant. For a reduced structure \eqref{eq2.6} with $P=[P_1,\ldots,P_g]$, $R=[R_1,\ldots,R_g]$ and $Q=[Q_1,\ldots,Q_g]$ partitioned conformally with the diagonal blocks, so that $B_1^{(t)}=P_t^*B_1R_t$ and $A_{11}^{(t)}=P_t^*A_{11}Q_t$, we call $\mathcal{U}_t=\mathcal{R}(P_t)$ the \emph{left subspace} of the component $[B_1^{(t)}\mid A_{11}^{(t)}]$. Two components $[B_1^{(t)}\mid A_{11}^{(t)}]$ and $[B_1^{(s)}\mid A_{11}^{(s)}]$, possibly of different reduced structures, are called \emph{equivalent} if there are unitary matrices $\tilde P,\tilde R,\tilde Q$ such that $\tilde P^*[B_1^{(t)}\tilde R\mid A_{11}^{(t)}\tilde Q]=[B_1^{(s)}\mid A_{11}^{(s)}]$.

\begin{lemma}\label{le:structures and invariant subspaces}
Let $P=[P_1,\ldots,P_g]$ be unitary.
\begin{enumerate}
\item There are unitary matrices $R=[R_1,\ldots,R_g]$ and $Q=[Q_1,\ldots,Q_g]$, some of whose blocks may be empty, such that $P^*[B_1R\mid A_{11}Q]$ has the block diagonal form \eqref{eq2.6} with diagonal blocks $[P_t^*B_1R_t\mid P_t^*A_{11}Q_t]$, if and only if every $\mathcal{R}(P_t)$ is $\mathcal{A}$-invariant.
\item In this case every $B_1^{(t)}$ and every nonempty $A_{11}^{(t)}$ has full column rank, and
$P_t^*B_1B_1^*P_t=B_1^{(t)}B_1^{(t)*}$, $P_t^*A_{11}A_{11}^*P_t=A_{11}^{(t)}A_{11}^{(t)*}$.
\item A component $[B_1^{(t)}\mid A_{11}^{(t)}]$ is irreducible if and only if its left subspace $\mathcal{R}(P_t)$ is an irreducible $\mathcal{A}$-invariant subspace, that is, contains no nonzero proper $\mathcal{A}$-invariant subspace.
\end{enumerate}
\end{lemma}
\begin{proof}
For (1), if $P^*B_1R$ and $P^*A_{11}Q$ are block diagonal, then so are $P^*B_1B_1^*P=(P^*B_1R)(P^*B_1R)^*$ and $P^*A_{11}A_{11}^*P$, so every $\mathcal{R}(P_t)$ is invariant. Conversely, suppose that every $\mathcal{R}(P_t)$ is invariant, so that $P^*B_1B_1^*P$ is block diagonal, and let $D_t=P_t^*B_1$. Then $D_tD_s^*=0$ for $t\neq s$, so the subspaces $\mathcal{R}(D_t^*)$ are mutually orthogonal. Let the columns of $R_t$ form an orthonormal basis of $\mathcal{R}(D_t^*)$. Since $B_1$ has full column rank, $\sum_t\mathcal{R}(D_t^*)=\mathcal{R}(B_1^*P)$ is the whole space, so $R=[R_1,\ldots,R_g]$ is unitary, and $D_tR_s=0$ for $t\neq s$; hence $P^*B_1R={\rm diag}(D_1R_1,\ldots,D_gR_g)$. The same construction applied to $P^*A_{11}$ gives $Q$, where $Q_t$ is empty if $P_t^*A_{11}=0$.
For (2), $B_1R$ has full column rank, and after multiplication by $P^*$ it is block diagonal, so each diagonal block has full column rank; the same holds for $A_{11}$. Moreover $P_t^*B_1R_s=0$ for $s\neq t$, so $P_t^*B_1=P_t^*B_1R_tR_t^*$ and $P_t^*B_1B_1^*P_t=B_1^{(t)}B_1^{(t)*}$; similarly for $A_{11}$.
For (3), by (2) the matrices $B_1^{(t)}B_1^{(t)*}$ and $A_{11}^{(t)}A_{11}^{(t)*}$ are the compressions of $B_1B_1^*$ and $A_{11}A_{11}^*$ to the invariant subspace $\mathcal{R}(P_t)$. Hence, through $P_t$, the subspaces of $\mathbb{C}^{m_t}$ invariant under these compressions correspond to the $\mathcal{A}$-invariant subspaces contained in $\mathcal{R}(P_t)$. By (2), the component satisfies the full column rank assumptions used in (1), and (1) applied to the component shows that it is reducible if and only if $\mathbb{C}^{m_t}$ is a nontrivial orthogonal sum of subspaces invariant under these compressions; since orthogonal complements of invariant subspaces are invariant, this holds if and only if $\mathcal{R}(P_t)$ contains a nonzero proper $\mathcal{A}$-invariant subspace.
\end{proof}

Let $\hat A_1=[B_1\ \ 0]$ and $\hat A_2=[0\ \ A_{11}]$, with zero blocks of appropriate sizes. Since $\hat A_1\hat A_2^*=0$, the left $*$-algebra associated with $\{\hat A_1,\hat A_2\}$ in~\cite{MaeharaMurota2011LAA}, generated by the products $\hat A_i\hat A_j^*$, is $\mathcal{A}$.

\begin{proposition}\label{prop:ssvd}
Let $[B_1\mid A_{11}]$ be a core problem over $\mathbb{C}$.
\begin{enumerate}
\item For a unitary matrix $P=[P_1,\ldots,P_g]$, there are unitary matrices $R$ and $Q$ such that $P^*[B_1R\mid A_{11}Q]$ has the block diagonal form \eqref{eq2.6} with diagonal blocks $[P_t^*B_1R_t\mid P_t^*A_{11}Q_t]$ if and only if there is a unitary matrix $\hat Q$ such that $P^*\hat A_1\hat Q$ and $P^*\hat A_2\hat Q$ are block diagonal with respect to the row partition of $P$ and a common column partition. Hence the finest reduced structures of $[B_1\mid A_{11}]$ correspond to the finest simultaneous singular value decompositions of $\{\hat A_1,\hat A_2\}$, with the same left partitions.
\item Two components are equivalent if and only if the restrictions of $\mathcal{A}$ to their left subspaces $\mathcal{U}_t$ and $\mathcal{U}_s$ are unitarily equivalent, that is, $WMu=MWu$ for a unitary map $W:\mathcal{U}_t\to\mathcal{U}_s$, all $u\in\mathcal{U}_t$ and $M\in\{B_1B_1^*,A_{11}A_{11}^*\}$.
\end{enumerate}
\end{proposition}
\begin{proof}
For (1), if $P^*[B_1R\mid A_{11}Q]$ is block diagonal, let $\hat Q$ be ${\rm diag}(R,Q)$ with its columns reordered so that those of $R_t$ and $Q_t$ form the $t$th group. Conversely, if $P^*\hat A_1\hat Q$ and $P^*\hat A_2\hat Q$ are block diagonal, then so are $P^*B_1B_1^*P$ and $P^*A_{11}A_{11}^*P$, and \Cref{le:structures and invariant subspaces}(1) gives $R$ and $Q$. With \Cref{le:structures and invariant subspaces}(3), both kinds of finest decompositions therefore have as left partitions the orthogonal decompositions of $\mathbb{C}^m$ into irreducible $\mathcal{A}$-invariant subspaces.

For (2), by \Cref{le:structures and invariant subspaces}(2) the Gram matrices $B_1^{(t)}B_1^{(t)*}$ and $A_{11}^{(t)}A_{11}^{(t)*}$ of a component are the compressions of $B_1B_1^*$ and $A_{11}A_{11}^*$ to its invariant left subspace. If $\tilde P^*[B_1^{(t)}\tilde R\mid A_{11}^{(t)}\tilde Q]=[B_1^{(s)}\mid A_{11}^{(s)}]$, then $W=P_s\tilde P^*P_t^*$ intertwines. Conversely, $X=P_s^*WP_t$ is unitary with $XB_1^{(t)}B_1^{(t)*}X^*=B_1^{(s)}B_1^{(s)*}$, and similarly for $A_{11}$. For $N=XB_1^{(t)}$ and $M=B_1^{(s)}$, which have full column rank, $NN^*=MM^*$ gives $\mathcal{R}(N)=\mathcal{R}(M)$ and equal numbers of columns, so $V=M^+N$ is unitary with $N=MV$; the same holds for the $A_{11}$-parts when they are nonempty (otherwise take $V'=I_0$), and $\tilde P=X^*$, $\tilde R=V^*$, $\tilde Q=V'^*$ show the equivalence.
\end{proof}

By \Cref{prop:ssvd}(1), the existence and uniqueness of the finest reduced structure follow from~\cite[Thm.~2 and Cor.~3]{MaeharaMurota2011LAA} in the following sense: an irreducible reduced structure exists, and the number and sizes of its components are uniquely determined. By the structure of~\cite[Thm.~2(A)--(B)]{MaeharaMurota2011LAA}, the components fall into classes of mutually equivalent components, in the sense of \Cref{prop:ssvd}(2). For each class, the number $\mu_c$ of components and the sum of their left subspaces are the same for every irreducible reduced structure, whereas, if $\mu_c\ge2$, the left subspaces of the individual components of the class are not unique: in the block form $P^*\mathcal{A}P=\bigoplus_c\big(\mathbb{C}^{n_c\times n_c}\otimes I_{\mu_c}\big)$ of~\cite[Thm.~2.1 and (2.1)]{MaeharaMurota2011}, they correspond to the subspaces $\mathbb{C}^{n_c}\otimes{\rm span}\{x\}$ for an orthonormal basis $\{x\}$ of $\mathbb{C}^{\mu_c}$.

\begin{remark}\label{rem:commutant}
In the same basis, the commutant of $\mathcal{A}$ is $\bigoplus_c\big(I_{n_c}\otimes\mathbb{C}^{\mu_c\times\mu_c}\big)$~\cite{MaeharaMurota2011}, so that its complex dimension is $\sum_c\mu_c^2$; this identity is used in \Cref{sec:example}. The simultaneous singular value decomposition is also available over $\mathbb{R}$~\cite[Thm.~7 and Cor.~8]{MaeharaMurota2011LAA}, with algorithms that compute it from decompositions of the associated $*$-algebras~\cite{MurotaKannoKojimaKojima2010,MaeharaMurota2010,MaeharaMurota2011}. The present framework reaches the same decomposition through the row blocks, their interaction graph and the spectra of cycle products; the following results describe, in these terms, how multiplicities arise.
\end{remark}

\begin{proposition}\label{prop:complete division normal form}
Let $U_1,\ldots,U_{k+1}$ have orthonormal columns such that $U=[U_1,\ldots,U_{k+1}]$ has orthonormal columns and $\mathcal{H}=\mathcal{R}(U)$ is $\mathcal{A}$-invariant, and such that each $\mathcal{R}(U_j)$ is contained in the left singular subspace of $A_{11}$ for a singular value $\sigma_j$, or in $\mathcal{N}(A_{11}^*)$ with $\sigma_j=0$. Suppose that the row separation problem associated with $\{D_1,\ldots,D_{k+1}\}$, $D_j=U_j^*B_1$, is a complete division problem with $r$ rows per block, and let $H_1,\ldots,H_{k+1}$ be unitary matrices such that $\overline{H_iD_i}(H_jD_j)^T=\sigma_{(i,j)}I_r$ for all $i,j$ (by \Cref{thm:relation between eigen/singular}(ii) when $k\ge2$; for $k=1$, normalize the single nonzero cross-product directly, while for $k=0$ there is nothing to normalize, since $\bar D_1D_1^T$ is a scalar matrix). For $p=1,\ldots,r$ put
\[
u_{j,p}=U_jH_j^*e_p,\qquad \mathcal{L}_p={\rm span}\{u_{1,p},\ldots,u_{k+1,p}\}.
\]
Then:
\begin{enumerate}
\item the subspaces $\mathcal{L}_1,\ldots,\mathcal{L}_r$ are mutually orthogonal, $\mathcal{H}=\mathcal{L}_1\oplus\cdots\oplus\mathcal{L}_r$, and each $\mathcal{L}_p$ is $\mathcal{A}$-invariant;
\item in the bases $(u_{j,p})_j$, the compressions of $A_{11}A_{11}^*$ and $B_1B_1^*$ to $\mathcal{L}_p$ are ${\rm diag}(\sigma_j^2)$ and $\big[\sigma_{(j,i)}\big]_{i,j}$, independently of $p$; hence the restrictions of $\mathcal{A}$ to $\mathcal{L}_1,\ldots,\mathcal{L}_r$ are unitarily equivalent;
\item if the $\sigma_j$ are pairwise distinct, every $\mathcal{L}_p$ is irreducible, so that $\mathcal{H}$ is the orthogonal sum of $r$ mutually equivalent irreducible invariant subspaces.
\end{enumerate}
In particular, if the row separation problem constructed from the core problem (\Cref{def:row separation problem constructed by core problem}) is a complete division problem with $r$ rows per block, the core problem has exactly one class of irreducible components, with multiplicity $r$.
\end{proposition}
\begin{proof}
The vectors $u_{j,p}$ are orthonormal because $U$ has orthonormal columns and the $H_j$ are unitary, and they span $\mathcal{H}$. Each $u_{j,p}$ is a left singular vector, so $A_{11}A_{11}^*u_{j,p}=\sigma_j^2u_{j,p}$. Moreover, $u_{i,q}^*B_1=e_q^TH_iD_i$ is the $q$th row of $H_iD_i$, so
\[
\langle B_1B_1^*u_{j,p},u_{i,q}\rangle=\big(u_{i,q}^*B_1\big)\big(u_{j,p}^*B_1\big)^*=\big(H_iD_i(H_jD_j)^*\big)_{qp}=\sigma_{(j,i)}\delta_{pq},
\]
using $(H_iD_i)(H_jD_j)^*=\overline{\sigma_{(i,j)}}I_r=\sigma_{(j,i)}I_r$. Since $\mathcal{H}$ is invariant, $B_1B_1^*u_{j,p}\in\mathcal{H}$, and therefore $B_1B_1^*u_{j,p}=\sum_i\sigma_{(j,i)}u_{i,p}\in\mathcal{L}_p$. This proves (1) and the compressions in (2); identical compressions in corresponding bases mean that the unitary map $u_{j,p}\mapsto u_{j,p'}$ intertwines the restrictions to $\mathcal{L}_p$ and $\mathcal{L}_{p'}$.
For (3), let $\mathcal{V}\subseteq\mathcal{L}_p$ be a nonzero invariant subspace. Since the $\sigma_j$ are distinct, $\mathcal{V}$ is invariant under the spectral projectors of $A_{11}A_{11}^*$, so $\mathcal{V}={\rm span}\{u_{j,p}:\ j\in J\}$ for some nonempty $J$. If $j\in J$ and $\sigma_{(j,i)}\neq0$, then $u_{i,p}\in\mathcal{V}$ by the formula for $B_1B_1^*u_{j,p}$. The graph of nonzero cross-products is connected (\Cref{def:2d}), so $J=\{1,\ldots,k+1\}$ and $\mathcal{V}=\mathcal{L}_p$. The final statement follows, since in \Cref{def:row separation problem constructed by core problem} the singular values are distinct and $\mathcal{H}=\mathbb{C}^m$, together with \Cref{prop:ssvd}(2) and the discussion following \Cref{prop:ssvd}.
\end{proof}

\begin{corollary}\label{cor:degenerate}
An irreducible degenerate component of a core problem has size $1\times1$; by the choice of $R$, it can be taken to be a positive real number.
\end{corollary}
\begin{proof}
If $A_{11}^{(t)}$ has no columns, then $P_t^*A_{11}=0$, so $\mathcal{U}_t\subseteq\mathcal{N}(A_{11}^*)$ and $A_{11}A_{11}^*$ vanishes on $\mathcal{U}_t$. On $\mathcal{U}_t$, $\mathcal{A}$ therefore acts through polynomials in the Hermitian matrix $B_1B_1^*$, whose irreducible invariant subspaces are one-dimensional. By \Cref{le:structures and invariant subspaces}(3), $\dim\mathcal{U}_t=1$, say $\mathcal{U}_t={\rm span}\{u\}$, so $B_1^{(t)}$ is a single row. Moreover, for an orthonormal basis $U_0$ of $\mathcal{N}(A_{11}^*)$, the core problem property that $U_0^*B_1$ has full row rank implies $B_1^*u\neq0$. Since $u^*B_1=u^*B_1R_tR_t^*$ by \Cref{le:structures and invariant subspaces}(2), the block $R_t$ is nonempty, so $B_1^{(t)}$ has at least one column; being a single row of full column rank, it has exactly one column. Thus $B_1^{(t)}$ is a nonzero $1\times1$ matrix, and multiplying $R_t$ by a suitable phase makes it positive.
\end{proof}

\begin{remark}\label{rem:interpretation}
All splittings in \Cref{alg:strategy} multiply the row blocks from the left by unitary matrices and separate rows, so the strategy brings the row separation problem, through equivalent replacements, into the row block composition of two kinds of subproblems: those returned by Step~1.3, each with exactly one indivisible subspace, and the complete division problems returned by the output step, whose row spaces are mutually orthogonal. Via the correspondence between components and indivisible subspaces~\cite[Thm.~3.48]{Possible} (see Remark~\ref{re:possible structure}), the non-uniqueness associated with multiplicity is reflected by ambiguous indivisible subspaces in the row separation framework, and complete division problems provide the explicit form in which this ambiguity occurs in \Cref{alg:strategy}; \Cref{prop:complete division normal form} describes this form. In the example of \Cref{sec:example}, the equivalent pairs of components arising from $\mathcal{G}_1$ and $\mathcal{G}_3$ are returned by the output step, and all other components are returned by Step~1.3.
\end{remark}

\subsection{Numerical Example}\label{sec:example}
The following example illustrates that the strategy of \Cref{alg:strategy} recovers the finest decomposition over $\mathbb{C}$
when its exact tests are replaced by comparisons with a tolerance; it is not intended as a
numerical algorithm. A real core problem with $m=40$, $d=22$ and $n=32$ is assembled
block diagonally from the six subproblems listed in \Cref{tab:fp-design}, whose explicit matrices are
given in \Cref{sec:app eg}. The four nonzero singular values $1,2,3,4$ of $A$ have multiplicities $9,9,8,6$, respectively, while $\mathcal{N}(A^*)$ has dimension $8$. Thus none of the nonzero singular values is simple, precisely the case not covered by the strategy of~\cite{Possible}. It is hidden as $PAQ$ and $PBR$ with random real orthogonal
$P$, $Q$, $R$ (generated as described in \Cref{sec:app eg}), and the strategy is applied to these matrices in complex arithmetic with the
tolerance $\tau=10^{-8}$. The Python implementation, which follows \Cref{alg:strategy} step by step, and the data are available as supplementary material.

\begin{table}[ht]
\centering
\caption{Planted subproblems and the components recovered by the strategy. Sizes are given as $B^{(t)}\mid A^{(t)}$. The third column gives the number of components of the finest decomposition over $\mathbb{R}$ and over $\mathbb{C}$ (``equiv.'': the components over $\mathbb{C}$ are mutually equivalent); the fourth column gives the size of each recovered component; in the fifth column, $c$ is the length of the cycle split before the components are returned; the last column gives their behaviour under complex conjugation.}
\label{tab:fp-design}
\small
\setlength{\tabcolsep}{3pt}
\begin{tabular}{llllll}
\toprule
 & Planted & No.\ $\mathbb{R}\to\mathbb{C}$ & Recovered & Returned by & Conjugation \\
\midrule
$\mathcal{G}_1$ & $16\times8\mid16\times16$ & $1\to2$, equiv. & $8\times4\mid8\times8$ & output, $c=3$ & span invariant \\
$\mathcal{G}_2$ & $8\times6\mid8\times6$   & $1\to2$ & $4\times3\mid4\times3$ & Step~1.3, $c=4$ & conjugate pair \\
$\mathcal{G}_3$ & $8\times4\mid8\times6$   & $2\to2$, equiv. & $4\times2\mid4\times3$ & output & span invariant \\
$\mathcal{G}_4$ & $3\times1\mid3\times2$   & $1\to1$ & $3\times1\mid3\times2$ & Step~1.3 & self-conjugate \\
$\mathcal{G}_5$ & $3\times1\mid3\times2$   & $1\to1$ & $3\times1\mid3\times2$ & Step~1.3 & self-conjugate \\
$\mathcal{G}_6$ & $2\times2\mid2\times0$   & $2\to2$ & $1\times1\mid1\times0$ & Step~1.3 & self-conjugate \\
\midrule
Total & $40\times22\mid40\times32$ & $8\to10$ & & & \\
\bottomrule
\end{tabular}
\end{table}

\begin{table}[ht]
\centering
\caption{Decisions of the strategy on $PBR$ and $PAQ$, rounded to ten decimal digits, with $\tau=10^{-8}$: for each kind of
decision, the number of decisions, the largest quantity classified as zero and the smallest
quantity classified as nonzero.}
\label{tab:fp-results}
\small
\begin{tabular}{lccc}
\toprule
Decision & Count & Largest zero & Smallest nonzero \\
\midrule
Singular values of $A$               & 44   & $1.4\times10^{-11}$ & $2.5\times10^{-1}$ \\
Step~1.1 (equal singular values)          & 35   & $5.3\times10^{-11}$ & $2.1\times10^{-1}$ \\
Step~1.3 (rank of the minimized subspace) & 234  & $1.2\times10^{-9}$  & $1.7\times10^{-1}$ \\
Step~2.1 (vanishing cross-product)        & 326  & $7.6\times10^{-10}$ & $5.0\times10^{-1}$ \\
Step~2.1 (scalar multiple of unitary)     & 72   & $1.1\times10^{-9}$  & $1.0$ \\
Step~2.2 (equal singular values)          & 26   & $1.6\times10^{-10}$ & $1.0$ \\
Step~c.1 (scalar covariance matrix)       & 8038 & $2.7\times10^{-10}$ & $7.7\times10^{-2}$ \\
Step~c.2 (equal eigenvalues)              & 7    & $5.9\times10^{-10}$ & $9.7\times10^{-1}$ \\
\midrule
All decisions                             & 8782 & $1.2\times10^{-9}$  & $7.7\times10^{-2}$ \\
\bottomrule
\end{tabular}

\end{table}

The strategy returns ten components, whose sizes and structure are listed in \Cref{tab:fp-design}, and it exercises every step of \Cref{alg:strategy}: Steps~1.1--1.3; Steps~2.1--2.3, with two mutually orthogonal groups in Step~2.3; a splitting at $c=3$ (for $\mathcal{G}_1$) and at $c=4$ (for $\mathcal{G}_2$); and the output step for the two complete division problems arising from $\mathcal{G}_1$ and $\mathcal{G}_3$. Let $\hat P$ collect the orthonormal left bases of the recovered components, and let $\hat R$ and $\hat Q$ collect orthonormal bases of the row spaces of their $B$- and $A$-parts. The normalized off-block residual, that is, the Frobenius norm of the off-diagonal blocks of $\hat P^*[PBR\hat R\mid PAQ\hat Q]$ divided by $\|[B\mid A]\|_F$, equals $4.0\times10^{-10}$.

The recovered decomposition can be checked independently of the strategy. By \Cref{sec:ssvd}, the finest decomposition is the finest joint block diagonalization of the Hermitian pair $\{\tilde B\tilde B^T,\tilde A\tilde A^T\}$ formed from the hidden data $\tilde B=PBR$ and $\tilde A=PAQ$, and by Remark~\ref{rem:commutant} the commutant of this pair has complex dimension $\sum_c\mu_c^2$, where $\mu_c$ are the multiplicities of the classes of unitarily equivalent irreducible components. Computed from the data alone, the commutant has numerical complex dimension $14$: the relative singular values of the corresponding linear map are either below $4.1\times10^{-11}$ or at least $1.4\times10^{-2}$. For the recovered components, the spaces of intertwiners, computed in the same way, have dimension one for each component with itself, so that every component is irreducible; dimension one in both directions within exactly the two equivalent pairs arising from $\mathcal{G}_1$ and from $\mathcal{G}_3$; and dimension zero otherwise. The sum of these dimensions is $14$, in agreement with the commutant. An analogous independent computation, in which the eigenspaces of a generic symmetric (respectively Hermitian) element of the commutant computed over $\mathbb{R}$ (respectively $\mathbb{C}$) give the irreducible components, confirms the eight real and ten complex components of \Cref{tab:fp-design}.

Since the data are real, the finest decomposition over $\mathbb{C}$ refines the finest decomposition under real orthogonal transformations, and \Cref{tab:fp-design} shows where the two differ. The subproblems $\mathcal{G}_1$ and $\mathcal{G}_2$ are irreducible over $\mathbb{R}$ but split over $\mathbb{C}$: $\mathcal{G}_2$ into a pair of complex conjugate components, as in Example~\ref{eg:4-cycle}, and $\mathcal{G}_1$ into two equivalent components. The subproblems $\mathcal{G}_3$ and $\mathcal{G}_6$ consist of two components already over $\mathbb{R}$. Thus the decomposition over $\mathbb{C}$ has ten components, whereas the finest real one has eight. The decomposition over $\mathbb{C}$ is closed under complex conjugation: each component is either self-conjugate or the conjugate of another component, and for the pairs of equivalent components only their combined span is determined. The last column of \Cref{tab:fp-design} records this, and the corresponding defects are at most $2.7\times10^{-9}$.

\Cref{tab:fp-results} lists all $8782$ decisions. After Step~1.1 every block is scaled so that $\bar D_jD_j^T=I$ (Remark~\ref{rem:fp scaling}), and all tested quantities are dimensionless: relative gaps between singular values or eigenvalues, spectral norms of cross-products, the relative spread of the singular values of cross-products, and the distance of covariance matrices from scalar matrices. The largest quantity classified as zero is $1.2\times10^{-9}$ and the smallest classified as nonzero is $7.7\times10^{-2}$; hence every tolerance $\tau$ with $1.2\times10^{-9}\le\tau<7.7\times10^{-2}$ yields exactly the same decisions. \Cref{tab:fp-results} refers to the hidden matrices rounded to ten decimal digits, which perturbs them by at most $5\times10^{-11}$ and accounts for the size of the quantities classified as zero. With the full-precision matrices the strategy makes the same decisions and returns the same components, the quantities classified as zero are below $3\times10^{-14}$, and the off-block residual is about $5\times10^{-15}$. The same decisions and components are obtained for the eight further random seeds $1,\dots,8$.

\section{Conclusions}\label{sec:con}
This paper develops a strategy that computes the finest decomposition of a core problem under unitary transformations, together with the theory of division problems on which it rests. The strategy requires no prior information on the number or the sizes of the blocks, and it is guaranteed to return the finest decomposition, so that none of the returned components can be decomposed further. Moreover, we relate the decomposition to the graph induced by the row blocks, which reduces the decomposition problem to the search for chordless cycles whose covariance matrices are not scalar matrices. The resulting finest reduced structure corresponds to a finest simultaneous singular value decomposition in the sense of Maehara and Murota~\cite{MaeharaMurota2011LAA}, from whose theory its existence and its uniqueness, in the sense of \Cref{sec:ssvd}, follow; in the row separation framework, multiplicities are represented explicitly by complete division problems.

Several directions remain open for the present row separation approach. Although the finest decomposition under real orthogonal transformations is already provided by the simultaneous singular value decomposition theory of Maehara and Murota, a real-arithmetic version of the present strategy with error-controlled numerical decisions is a natural next step.
 Since the present strategy is formulated in exact arithmetic and is not optimized for computational cost, possible efficiency improvements include searching directly for chordless cycles and exploiting the parallelism of the individual steps.
 
 On the side of the TLS theory, the components provide a finer unit for solvability analysis, and the solvability of individual components and its relation to the solution of the original problem deserve further study. Finally, the graph induced by the row blocks provides a systematic way to construct problems with prescribed cycles, coupling and eigenvalue multiplicities; such problems can serve to validate algorithms for simultaneous block diagonalization.

\section{Appendix}\label{sec:appendix}

   \begin{theorem}[Existence of Minimized Subspace over $\mathbb{C}$]\label{thm:extend of existence of minimized}
Consider a row separation problem associated with $\{D_1,\ldots,D_{k+1}\}$, whose blocks have full row rank. For every nonzero vector $v\in\mathcal{V}=\sum_{j=1}^{k+1}\mathcal{R}\left(D_j^T\right)$, the minimized subspace $\mathcal{M}(v)$ exists and is unique.
\end{theorem}
\begin{proof}
Let $K_j=D_j^T\bar D_j$, $j=1,\ldots,k+1$, be the row-space Gram operators, and define
\[
\mathcal{M}^{(0)}={\rm span}\{v\},\qquad
\mathcal{M}^{(q+1)}=\mathcal{M}^{(q)}+\sum_{j=1}^{k+1}K_j\mathcal{M}^{(q)},\qquad q=0,1,\ldots
\]
Since $v\in\mathcal{V}$ and $\mathcal{R}(K_j)=\mathcal{R}(D_j^T)\subseteq\mathcal{V}$, every $\mathcal{M}^{(q)}$ is contained in $\mathcal{V}$. The sequence is increasing, and its dimension can increase strictly only finitely many times, so there is a $q$ with $\mathcal{M}^{(q+1)}=\mathcal{M}^{(q)}=:\mathcal{M}$. Then $K_j\mathcal{M}\subseteq\mathcal{M}$ for every $j$, and $\mathcal{M}$ is a nonzero subspace of $\mathcal{V}$; by \Cref{le:D_j and S cap}, $\mathcal{M}$ is a row separation subspace containing $v$.

Let $\mathcal{S}$ be any row separation subspace with $v\in\mathcal{S}$. By \Cref{le:D_j and S cap}, $K_j\mathcal{S}\subseteq\mathcal{S}$ for every $j$. Then $\mathcal{M}^{(0)}\subseteq\mathcal{S}$, and $\mathcal{M}^{(q)}\subseteq\mathcal{S}$ implies $\mathcal{M}^{(q+1)}\subseteq\mathcal{S}$; hence $\mathcal{M}\subseteq\mathcal{S}$. Thus $\mathcal{M}$ satisfies the three conditions of \Cref{def:minimized subspace}, and it is unique, since two subspaces satisfying them are contained in each other.
\end{proof}

\begin{remark}
The minimized subspace $\mathcal{M}(v)$ is the smallest subspace containing $v$ that is invariant under all row-space Gram operators $D_j^T\bar D_j$, and the construction above is the finite closure of $v$ under these operators. Invariance under the orthogonal projectors onto the $\mathcal{R}(D_j^T)$ would not suffice in general (see the remark following \Cref{le:D_j and S cap}).
\end{remark} 

We conclude with a worked example of \Cref{alg:strategy} and the details of the numerical example of \Cref{sec:example}.
\begin{example}\label{eg:C+R}
    Consider the real core problem with $A\in\mathbb{R}^{7\times 6}$ and $B\in\mathbb{R}^{7\times 5}$,
\[
B=
\begin{bmatrix}
1 & 0 & 0 & 0&0\\
0 & 1 & 0 & 0&0\\
0 & -1 & -1 & -1&0\\
1 & 0 & 1 & -1&0\\
0 & -1 & 0 & 1&0\\
1 & 0 & -1 & 0&0\\
0&0&0&0&1
\end{bmatrix},
A=
\begin{bmatrix}
1 & 0 & 0 & 0&0&0\\
0 & 1 & 0 & 0&0&0\\
0 & 0 & 2 & 0&0&0\\
0 & 0 & 0 & 2&0&0\\
0 &0 & 0 & 0&3&0\\
0 & 0 & 0 & 0&0&3\\
0&0&0&0&0&0\\
\end{bmatrix}.
\]
\end{example}
We apply the strategy of \Cref{sec:final strategy}:
\begin{enumerate}
    \item \emph{Input.} The matrix $A$ has three distinct nonzero singular values, each of multiplicity two, and $\mathcal{N}(A^{T})$ has dimension
    one. Let $U_1,U_2,U_3\in\mathbb{R}^{7\times 2}$ have orthonormal columns spanning the corresponding left singular subspaces of $A$, let $U_4\in\mathbb{R}^{7\times 1}$ span $\mathcal{N}(A^{T})$, and define the
    row blocks
    $
    D_j = U_j^{T}B \in \mathbb{R}^{2\times 5},\ j = 1,2,3,
    \
    D_4 = U_4^{T}B \in \mathbb{R}^{1\times 5}.
    $

    \item \emph{1-division cycle.} The block $D_4$ consists of a single row, so Step~1.3 identifies it as an unambiguous indivisible direction and extracts it. The corresponding component has no columns of $A$, that is, it is a degenerate component.

    \item \emph{2-division cycle.} The remaining blocks $D_1,D_2,D_3$ form a $2$-division problem, so no splitting occurs in Steps~2.1--2.3.

    \item \emph{3-division cycle.} In Step~3.1 (Step~c.1 with $c=3$), the
    covariance matrix $\mathrm{Cov}^{(1,2,3)}$ has a pair of complex conjugate eigenvalues, each of multiplicity one. Step~3.2 therefore splits $D_1$, $D_2$ and $D_3$ into the six complex row blocks $D_1^{1}, D_1^{2}, D_2^{1}, D_2^{2}, D_3^{1}, D_3^{2}\in\mathbb{C}^{1\times 5}$,
    and the problem associated with these blocks is passed to Step~1.1 as a new input.

    \item \emph{Termination.} Every block now consists of a single row, and Step~1.3 returns
    $\{D_1^{1}, D_2^{1}, D_3^{1}\}$ and $\{D_1^{2}, D_2^{2}, D_3^{2}\}$, respectively. 
\end{enumerate}

\subsection{Numerical Example Details}\label{sec:app eg}
The random orthogonal matrices were generated in NumPy~2.3.5 using
\texttt{numpy.random.default\_rng} with seed \(20261015\). Each matrix was obtained by QR factorization of a standard Gaussian matrix, with column signs chosen so that the diagonal of the triangular factor was positive. The matrices were generated in the order \(P\), \(Q\), and \(R\).

\paragraph{$\mathcal{G}_1$} ($B^{(1)}\in\mathbb{R}^{16\times 8}$, $A^{(1)}\in\mathbb{R}^{16\times16}$)
\[
B^{(1)}=\begin{bmatrix}I_4&E\\I_4&J\\I_4&F\\I_4&0\end{bmatrix},\qquad
A^{(1)}=\operatorname{diag}(1,2,3,4)\otimes I_4,
\]
where
\[
E=\begin{bmatrix}0&-1&0&0\\1&0&0&0\\0&0&0&-1\\0&0&1&0\end{bmatrix},\quad
J=\begin{bmatrix}0&0&-1&0\\0&0&0&1\\1&0&0&0\\0&-1&0&0\end{bmatrix},\quad
F=\begin{bmatrix}0&0&0&-1\\0&0&-1&0\\0&1&0&0\\1&0&0&0\end{bmatrix}.
\]

\paragraph{$\mathcal{G}_2$} ($B^{(2)}\in\mathbb{R}^{8\times 6}$, $A^{(2)}\in\mathbb{R}^{8\times6}$)
\[
B^{(2)}=\begin{bmatrix}
I_2&0&0\\
\alpha I_2&\alpha I_2&\beta I_2\\
0&I_2&0\\
\alpha I_2&\alpha R(\theta_1)&\beta R(\theta_2)
\end{bmatrix},\qquad
A^{(2)}=\begin{bmatrix}\operatorname{diag}(1,2,3)\otimes I_2\\ 0_{2\times6}\end{bmatrix},
\]
where $\alpha=\sqrt{2/5}$, $\beta=\sqrt{1/5}$,
\[
R(\theta)=\begin{bmatrix}\cos\theta&\sin\theta\\-\sin\theta&\cos\theta\end{bmatrix},
\]
and $\cos\theta_1=-\tfrac78$, $\sin\theta_1=\tfrac{\sqrt{15}}{8}$, $\cos\theta_2=-\tfrac14$, $\sin\theta_2=-\tfrac{\sqrt{15}}{4}$.

\paragraph{$\mathcal{G}_3$} ($B^{(3)}\in\mathbb{R}^{8\times 4}$, $A^{(3)}\in\mathbb{R}^{8\times6}$)
\[
B^{(3)}=\begin{bmatrix}B_c&0\\0&B_c\end{bmatrix},\quad
A^{(3)}=\begin{bmatrix}A_c&0\\0&A_c\end{bmatrix},\qquad
B_c=\begin{bmatrix}1&0\\1&1\\0&1\\1&-1\end{bmatrix},\quad
A_c=\begin{bmatrix}1&0&0\\0&2&0\\0&0&3\\0&0&0\end{bmatrix}.
\]

\paragraph{$\mathcal{G}_4$} ($B^{(4)}\in\mathbb{R}^{3\times 1}$, $A^{(4)}\in\mathbb{R}^{3\times2}$)
\[
B^{(4)}=\begin{bmatrix}1\\1\\1\end{bmatrix},\qquad
A^{(4)}=\begin{bmatrix}4&0\\0&1\\0&0\end{bmatrix}.
\]

\paragraph{$\mathcal{G}_5$} ($B^{(5)}\in\mathbb{R}^{3\times 1}$, $A^{(5)}\in\mathbb{R}^{3\times2}$)
\[
B^{(5)}=\begin{bmatrix}1\\2\\1\end{bmatrix},\qquad
A^{(5)}=\begin{bmatrix}4&0\\0&2\\0&0\end{bmatrix}.
\]

\paragraph{$\mathcal{G}_6$} (two components, each with $A$ having no columns)
\[
B^{(6a)}=\begin{bmatrix}1\end{bmatrix},\qquad
B^{(6b)}=\begin{bmatrix}2\end{bmatrix}.
\]

\section*{Acknowledgments}
Bruno Carpentieri's work is supported by the European Regional Development and Cohesion Funds (ERDF) 2021–2027 under Project AI4AM - EFRE1052. He is a member of the \textit{Gruppo Nazionale per il Calcolo Scientifico} (GNCS) of the \textit{Istituto Nazionale di Alta Matematica} (INdAM). Yan-Fei Jing's work  is  supported by National Natural Science Foundation of China (12671442).

This research is funded by the European Social Fund Plus
Project code ESF2\_f3\_0003 CUP: B56F24000110001
“Excellence Scholarships for PhD students on topics of strategic relevance for South Tyrol”. 
 
 The authors would like to thank the handling editor and the anonymous reviewers for their constructive suggestions and valuable comments. Their suggestions and comments  helped us to substantially improve the readability and quality of this paper. The authors would like to thank Dr.~M.~Ple\v{s}inger (Technical University of Liberec) for helpful comments and insightful suggestions on the manuscript. The scientific ideas, theoretical framework, principal results, and decomposition strategy of this paper were developed by the authors. During the preparation and revision of the manuscript, generative AI tools (ChatGPT and Claude) were used to assist with language editing; the development, refactoring, and checking of computational code; the design and verification of the numerical example; and the drafting, revision, and checking of parts of the mathematical exposition, including some statements and proofs. All AI-assisted material was critically reviewed, verified, and, where necessary, corrected by the authors. The authors assume responsibility for all content.

\bibliographystyle{amsplain}
\bibliography{reference}

@Book{EDS.Total,
  editor =	 {S. {Van Huffel} and P. Lemmerling},
  title =	 {Total Least Squares and Errors-in-Variables Modeling: Analysis, Algorithms and Applications},
  publisher =	 {Kluwer Academic Publishers},
  address =	 {},
  year =	 2002,
  edition =	 {}
}

@Book{ref4,
  author =	 {G. H. Golub and C. F. {Van Loan}},
  title =	 {Matrix Computations},
  publisher =	 {The Johns Hopkins University Press},
  address =	 {Baltimore},
  year =	 2013,
  edition =	 {4th}
}

@Article{Ananalysisof,
  title =	 {An Analysis of the Total Least Squares problem},
  author =	 {G. H. Golub and C. F. {Van Loan}},
  journal =	 {SIAM J. Numer. Anal.},
  volume =	 {17},
  number     = {6},
  year =	 {1980},
  pages =	 {883-893},
}

@Article{Estimation,
  title =	 {Estimation in a Multivariate ``Errors in Variables'' Regression Model: Large Sample Results},
  author =	 {L. Gleser},
  journal =	 {Ann. Stat.},
  volume =	 {9},
  number     = {1},
  year =	 {1981},
  pages =	 {24-44},
}

@Article{systempulsetransfer,
  title =	 {Estimation of a system pulse transfer function in the presence of noise},
  author =	 {M. Levin},
  journal =	 {IEEE Trans. Automat. Control},
  volume =	 {9},
  year =	 {1964},
  pages =	 {229-235},
}

@Article{someidentificationmethods,
  title =	 {On a priori error estimates of some identification methods},
  author =	 {M. Aoki and P. Yue},
  journal =	 {IEEE Trans. Automat. Control},
  volume =	 {15},
  year =	 {1970},
  pages =	 {541-548},
}

@Article{Exactmaximumlikelihood,
  title =	 {Exact maximum likelihood parameter estimation of superimposed exponential
signals in noise},
  author =	 {Y. Bresler and A. Macovski},
  journal =	 {IEEE Trans. Acoust., Speech, Signal Process.},
  volume =	 {34},
     number     = {5},
  year =	 {1986},
  pages =	 {1081-1089},
}

@Article{Somemodified,
  title =	 {Some modified matrix eigenvalue problems},
  author =	 {G. H. Golub},
  journal =	 {SIAM Rev.},
  volume =	 {15},
  year =	 {1973},
  pages =	 {318-334},
}

@article{Theanalysis,
   author = {M. Wei},
    title = {The analysis for the total least squares problem with more than one solution},
   journal = {SIAM J. Matrix Anal. Appl.},
    volume     = {13},
    number     = {},
    pages      = {746-763},
    year       = {1992}
}

@article{Scaledtotal,
   author = {C. C. Paige and Z. Strakoš},
    title = {Scaled total least squares fundamentals},
   journal = {Numer. Math.},
    volume     = {91},
    number     = {},
    pages      = {117-146},
    year       = {2002}
}

@article{Classicalworks,
   author = {I. Hnětynková and M. Plešinger and D. M. Sima and Z. Strakoš and S. {Van Huffel}},
    title = {The total least
squares problem in {$AX \approx B$}: A new classification with the relationship to the classical works},
   journal = {SIAM J. Matrix Anal. Appl.},
    volume     = {32},
    number     = {},
    pages      = {748-770},
    year       = {2011}
}

@article{Thecore,
   author = {I. Hnětynková and M. Plešinger and Z. Strakoš},
    title = {The core problem within a linear approximation problem {$AX \approx B$} with multiple right-hand sides},
   journal = {SIAM J. Matrix Anal. Appl.},
    volume     = {34},
    number     = {},
    pages      = {917-931},
    year       = {2013}
}

@article{BGGK,
   author = {I. Hnětynková and M. Plešinger and Z. Strakoš},
    title = {Band generalization of the {G}olub-{K}ahan bidiagonalization, generalized {J}acobi matrices, and the core problem},
   journal = {SIAM J. Matrix Anal. Appl.},
    volume     = {36},
    number     = {},
    pages      = {417-434},
    year       = {2015}
}

@article{TLSformulation,
   author = {I. Hnětynková and M. Plešinger and J. Žáková},
    title = {{TLS} formulation and core reduction for problems with structured right-hand sides},
   journal = {Linear Algebra Appl.},
    volume     = {555},
    number     = {},
    pages      = {241-265},
    year       = {2018}
}

@article{OnTLS,
   author = {I. Hnětynková and M. Plešinger and J. Žáková},
    title = {On {TLS} formulation and core reduction for data fitting with generalized models},
   journal = {Linear Algebra Appl.},
    volume     = {577},
    number     = {},
    pages      = {1-20},
    year       = {2019}
}

@article{KrylovSubspace,
   author = {I. Hnětynková and M. Plešinger and J. Žáková},
    title = {Krylov {S}ubspace approach to core problems within multilinear approximation problems: a unifying framework},
   journal = {SIAM J. Matrix Anal. Appl.},
    volume     = {44},
    number     = {},
    pages      = {53-79},
    year       = {2023}
}

@article{TLSsense,
   author = {I. Hnětynková and M. Plešinger and D. M. Sima},
    title = {Solvability of the core problem with multiple
right-hand sides in the {TLS} sense},
   journal = {SIAM J. Matrix Anal. Appl.},
    volume     = {37},
    number     = {},
    pages      = {861-876},
    year       = {2016}
}

@article{minimization,
   author = {I. Hnětynková and M. Plešinger and J. Žáková},
    title = {Solvability classes for core problems in matrix total least squares minimization},
   journal = {Applications of Mathematics},
    volume     = {64},
    number     = {},
    pages      = {103-128},
    year       = {2019}
}

@phdthesis{ReductionofData,
   author    = {M. Plešinger},
   title     = {The {T}otal {L}east {S}quares {P}roblem and {R}eduction of {D}ata in {$AX \approx B$}},
     school = {Technical University of Liberec, Liberec, Czech Republic},
   year      = {2008},
 
}

@phdthesis{AnalysisPropertiesandBehaviour,
   author    = {J. Žáková},
   title     = {The Core Problem—Analysis, Properties, and Behaviour},
     school = {Technical University of Liberec, Liberec, Czech Republic},
   year      = {2023},
}

@article{Algebraic,
   author = {M. Wei},
    title = {Algebraic relations between the total least squares and least squares problems with more than one solution},
    journal = {Numer. Math.},
    volume     = {62},
    number     = {},
    pages      = {123-148},
    year       = {1992}
}

@article{Thedata,
   author = {R. D. DeGroat and E. M. Dowling},
    title = {The data least squares problem and channel equalization},
   journal = {IEEE Trans.
Signal Process.},
    volume     = {41},
    number     = {1},
    pages      = {407-411},
    year       = {1993}
}

@article{Possible,
   author = {S. Yu and Y. Jing},
    title = {Possible Reduced Structure of the Core Problem within the Total Least Square Problem},
   journal = {SIAM J. Matrix Anal. Appl.},
    volume     = {46},
    number     = 2,
    pages      = {1616-1639},
    year       = {2025}
}

@article{OptimalBackward,
   author = {J. Shan and Y. Wei},
    title = {Optimal Backward Error of a Total Least Squares and Its Randomized Algorithms},
   journal = {SIAM J. Matrix Anal. Appl.},
    volume     = {46},
    number     = {3},
    pages      = {2116-2139},
    year       = {2025}
}

@article{MurotaKannoKojimaKojima2010,
   author  = {K. Murota and Y. Kanno and M. Kojima and S. Kojima},
   title   = {A numerical algorithm for block-diagonal decomposition of matrix {$*$}-algebras with application to semidefinite programming},
   journal = {Japan J. Indust. Appl. Math.},
   volume  = {27},
   number  = {1},
   pages   = {125-160},
   year    = {2010}
}

@article{MaeharaMurota2010,
   author  = {T. Maehara and K. Murota},
   title   = {A numerical algorithm for block-diagonal decomposition of matrix {$*$}-algebras with general irreducible components},
   journal = {Japan J. Indust. Appl. Math.},
   volume  = {27},
   number  = {2},
   pages   = {263-293},
   year    = {2010}
}

@article{MaeharaMurota2011LAA,
  author  = {Maehara, Takanori and Murota, Kazuo},
  title   = {Simultaneous singular value decomposition},
  journal = {Linear Algebra Appl.},
  volume  = {435},
  number  = {1},
  year    = {2011},
  pages   = {106--116},
  doi     = {10.1016/j.laa.2011.01.007}
}

@article{MaeharaMurota2011,
   author  = {T. Maehara and K. Murota},
   title   = {Algorithm for error-controlled simultaneous block-diagonalization of matrices},
   journal = {SIAM J. Matrix Anal. Appl.},
   volume  = {32},
   number  = {2},
   pages   = {605-620},
   year    = {2011}
}
\end{document}